\documentclass[11pt, a4paper]{article}

\usepackage[utf8]{inputenc}
\usepackage{amsmath}
\usepackage{amssymb}
\usepackage{amsthm}
\usepackage[margin=1in]{geometry}
\usepackage{hyperref}
\usepackage{graphicx}
\usepackage[backend=bibtex]{biblatex}
\usepackage[printonlyused]{acronym}
\usepackage{authblk}

\usepackage{pgfplots}
\usepackage{pgfplotstable}
\usepackage{todonotes}
\pgfplotsset{compat=1.18}

\usepgfplotslibrary{groupplots}
\usepgfplotslibrary{statistics}

\definecolor{steelblue}{RGB}{70,130,180}
\definecolor{coral}{RGB}{255,127,80}

\newtheorem{theorem}{Theorem}[section]
\newtheorem{definition}{Definition}[section]
\newtheorem{remark}{Remark}[section]
\newtheorem{lemma}{Lemma}[section]

\newtheorem{conjecture}{Conjecture}[section]

\newcommand{\E}{\mathbb{E}}
\newcommand{\R}{\mathbb{R}}
\newcommand{\Var}{\mathrm{Var}}
\newcommand{\Bias}{\mathrm{Bias}}
\newcommand{\MSE}{\mathrm{MSE}}
\newcommand{\MISE}{\mathrm{MISE}}
\newcommand{\Ocal}{\mathcal{O}}
\newcommand{\ee}{\mathrm{e}}
\newcommand{\dd}{\mathrm{d}}

\title{Data driven extreme value distribution estimation: \\ Derivation of the Mean Integrated Squared Error, optimal bandwidth selection and stability conditions}
\author[1]{Michael Sandbichler\thanks{Corresponding author: michael.sandbichler@datalabhell.at}}
\author[1]{Tobias Hell}
\affil[1]{Data Lab Hell GmbH, Zirl, Austria}
\date{\today}

\begin{document}

\maketitle
\begin{abstract}
We introduce the~\ac{ddevd} estimator, a kernel-based method for estimating extreme value distributions from data.
We derive its~\ac{mise} in detail, use it to compute the optimal bandwidth and establish stability conditions for the bandwidth optimization procedure.
\end{abstract}

\tableofcontents

\section{Introduction}

The statistical analysis of extreme events is a cornerstone of risk assessment in fields as diverse as finance, insurance, engineering, hydrology and climate science~\cite{embrechts2013modelling, mcneil2015quantitative, castillo2012extreme, katz2002statistics,schellander2018modeling}.
The primary goal of extreme value statistics is to quantify the probability of events that are more extreme than any previously observed, a task that requires careful extrapolation beyond the range of the data.
Foundational results from \ac{evt} provide the theoretical basis for such extrapolation.

A classical approach in \ac{evt} is the \ac{bmm}~\cite{gumbel1958statistics, ferreira2015block}, in which the observed period is divided into non-overlapping chunks and the maximum value within each chunk is recorded.
Then, by virtue of the Fisher-Tippett-Gnedenko theorem the distribution of these maxima converges to a \ac{gev} distribution. These distributions broadly fall into three
families: the Gumbel, Fréchet, and Weibull distributions. These distribution differ mainly in their tail behaviour, a fact which we will use later in our analysis.
The \ac{gev} distribution is characterized by three parameters: location, scale, and shape. The shape parameter, in \ac{evt} typically denoted by $\gamma$, is arguably
the most important one, as it determines the tail behavior of the distribution~\cite{haan2006extreme}.
Fitting this distribution to the observed data then allows us to make probabilistic statements about extreme events within a block. The problem with this approach is that
it requires a large number of blocks to obtain reliable estimates of the parameters and ideally a lot of data within each block to ensure that a reliable maximum is captured.
Then in the end a large amount of data is discarded - because for the fitting only a single value for each block is used. To illustrate this, let us consider an example from 
climate science. In order to obtain 'yearlyness' of extreme events (think 100-year events), one would typically have to analyze many years of data in order to obtain a reliable
estimate while throwing away 364 values for each year~\cite{schellander2019error}.

A more data efficient parametric approach is the \ac{pot} method, which uses all data above a certain threshold.
This method relies on the Balkema-de Haan-Pickands theorem, which shows that the distribution of the excesses over the threshold can be modeled by a
\ac{gpd} for a large class of base distributions~\cite{drees1998smooth,haan2006extreme,pickands1975statistical}. The \ac{gpd} is also characterized by a shape parameter that has a direct correspondence to the 
shape parameter of the \ac{gev} distribution.

Both, the \ac{bmm} and \ac{pot} methods remain standard tools in the analysis of extreme events.
However a key limitation of both is their reliance on the assumption that we operate in the asymptotic regime.
Their performance may be severely hampered if too few data points are available for the estimation of the extreme value parameters, a problem that can for example occur if the 
data in question has only been recorded for a relatively short period of time (which in climate science can be a few years).

A new method to tackle the mentioned problems is to use a metastatistical approach~\cite{marani2015metastatistical,miniussi2020metastatistical}.
The core idea is to model the base distribution and by using the classical property of the \ac{cdf} of 
the maximum over $n$ random variables with this distribution. Let us shortly outline this approach.
Suppose we are given $n$ i.i.d. observations of a random variable $X$ within a single block. Then the \ac{cdf} of the maximum of these observations is given by
\begin{equation}
    F_{\max}(y) = \left(F_X(y)\right)^n,
\end{equation}
where $F_X$ is the \ac{cdf} of the base distribution. This is perhaps the most classical result in extreme value theory, 
but the key insight is that this relationship allows us to derive the \ac{cdf} of the maximum directly from the \ac{cdf} of the base distribution.
So if we are given $m$ blocks with $n_i$ observations each for $i=1,\ldots,m$, we can use this relationship to estimate the extreme value distribution over the blocks via
\[F_{MEV}(y) := \frac1m \sum_{i=1}^m F_{\max,i}(y) = \frac1m \sum_{i=1}^m \left(F_{X,i}(y)\right)^{n_i},\]
where $F_{X,i}$ is the \ac{cdf} of the base distribution for block $i$.
This distribution is called the \emph{metastatistical extreme value distribution} and will be the main focus of this paper.
A drawback of this approach is that $F_{X,i}$ is typically chosen to be a parametric distribution, which may not accurately reflect the true underlying distribution of the data.

In this paper we introduce the \ac{ddevd} estimator as a non-parametric alternative that does not rely on such strong assumptions.
It can be seen as a nonparametric plug-in solution that uses a KDE-based approach to estimate the \ac{cdf} of the base distribution occurring in the definition of the \ac{mev} distribution.
We use the data points within each block to estimate the \ac{cdf} of the base distribution~\cite{nadaraya1964some, silverman2018density}.
In kernel density estimation the selection of the bandwidth is of crucial importance~\cite{sheather1991reliable, jones1996brief}, and we will also address this issue in the context of the \ac{ddevd} estimator.

The paper will be structured as follows. In the following section we will formally define the problem setting and the \ac{ddevd} estimator.
There we will also recall some classical definitions and results that will be useful later on.
In Section~\ref{sec:main_results} we will present the main results of our analysis. We will show asymptotic expansions for bias and variance and how
these can be used to derive a strategy to select the bandwidth of the estimator optimally. Further, we will examine the stability of this bandwidth
selection procedure depending on the tail behaviour of the base distribution. 
This will lead us to asymptotic constraints between the number of blocks and the number of data points within each block and we will see that the stability
of this approach depends fundamentally on the tail behaviour of the base distribution, which we will characterize by its extreme value index.
This analysis gives us a clear guideline as to when the \ac{ddevd} estimator can be readily applied and what its limitations are.
The main theorems will be stated in this section, but the majority of the proofs and derivations will be deferred to the appendix due to their technical nature.
Finally in Section~\ref{sec:conclusion} we will recap the findings of the paper and discuss potential future work.

\subsection{Problem Statement}

We consider the task of estimating a target~\ac{cdf}, $F(y)$, using a kernel-based method.
The structure of the problem is as follows. Suppose we are given $m$ blocks, each containing $n_i$ i.i.d. samples $X_{ij}$
drawn from a base distribution with~\ac{cdf} $F_X$. 
The goal is to estimate the target~\ac{cdf}
\[F(y) = \frac1m \sum_{i=1}^m F_X^{n_i}(y)\]
using a kernel estimator that aggregates the contributions from all blocks.

\begin{definition}[\ac{ddevd} Estimator]
    Let $X_{ij}$ be a random i.i.d. sample from a base distribution with~\ac{cdf} $F_X$, 
    $h = (h_{ij})_{i,j} > 0$ is the matrix of bandwidth parameters, and $m, n_i \in \mathbb{N}_{>0}$.
    Furthermore let $K$ be a~\ac{cdf} kernel function, which is a non-decreasing function with 
    $\lim_{y \to -\infty} K(y) = 0$ and $\lim_{y \to \infty} K(y) = 1$.
    We define the \emph{\ac{ddevd} estimator} $\hat F_h(y)$ as follows:
    \begin{equation}
        \hat F_h(y) = \frac{1}{m} \sum_{i=1}^m \left(\frac1{n_i}\sum_{j=1}^{n_i} K\left(\frac{y-X_{ij}}{h_{ij}}\right)\right)^{n_i}
    \end{equation}
    For notational convenience, we sometimes denote the $i$-th inner approximation as $F_{i,h}(y) = \frac1{n_i}\sum_{j=1}^{n_i} K\left(\frac{y-X_{ij}}{h_{ij}}\right)$,
    so that $\hat F_h(y) = \frac{1}{m} \sum_{i=1}^m F_{i,h}^{n_i}(y)$.
\end{definition}

The overall goal of this paper is to analyze the \ac{mise} of the estimator $\hat F_h(y)$.

\begin{definition}[Mean (Integrated) Squared Error]
    The \ac{mise} of the estimator $\hat F_h(y)$ is defined as:
    \begin{equation}
        \MISE(\hat F_h) := \E \int (\hat F_h(y) - F(y))^2 \mathrm{d}y
    \end{equation}
    This integral is taken over the entire domain of $y$ where the \ac{cdf} is defined.
    By interchanging the integration order, this term inside the integral is the \ac{mse} of the estimator at point $y$.
    \begin{equation}
        \MSE(y) = \E[(\hat F_h(y) - F(y))^2]
    \end{equation}
    Further, define the $q$-quantile version of the \ac{mise} ($q$-\ac{mise}) as follows:
    \begin{equation}
        \MISE_q(\hat F_h) := \E \int_{F^{-1}(q) }^\infty (\hat F_h(y) - F(y))^2 \mathrm{d}y
    \end{equation}
    where $F^{-1}(q)$ is the quantile function of $F$ at level $q$.
\end{definition}

The \ac{mise} is a common measure of the overall performance of an estimator, as it captures both bias and variance across the entire domain.
The $q$-quantile version of the \ac{mise} is particularly relevant in the context of extreme value estimation, as it focuses on the tail of the distribution, which is often of primary interest in applications.
In the following derivations, it will become clear that in our setting, the $q$-\ac{mise} is the relevant quantity to analyze.

Let us recall a standard result for the decomposition of the mean squared error of the estimator $\hat F_h(y)$:

\begin{theorem}[Bias-Variance Decomposition]\label{thm:bias_variance}
The \ac{mse} of the estimator $\hat F_h(y)$ can be decomposed into bias and variance components as follows:
\begin{equation}
    \MSE(y) = \Bias^2(\hat F_h)(y) + \Var(\hat F_h)(y)
\end{equation}
where the bias and variance are defined as:
\begin{align}
    \Bias(\hat F_h)(y) &:= \E[\hat F_h(y)] - F(y) \\
    \Var(\hat F_h)(y) &:= \E\left[(\hat F_h(y) - \E[\hat F_h(y)])^2\right]
\end{align}
\end{theorem}
\begin{proof}
The proof is a straightforward computation, we add and subtract $\E \hat F_h(y)$ within the expression for the \ac{mse}.
\begin{align*}
\MSE(y) &= \E[(\hat F_h(y) - F(y))^2] \\
        &=  \E[((\hat F_h(y) - \E \hat F_h(y)) + (\E \hat F_h(y) - F(y)))^2]\\
        &= \underbrace{\E[(\hat F_h(y) - \E \hat F_h(y))^2]}_{= \Var(\hat F_h)(y)} + 2(\E \hat F_h(y) - F(y))\underbrace{\E[\hat F_h(y) - \E \hat F_h(y)]}_{=0}  + (\underbrace{\E \hat F_h(y) - F(y)}_{=\Bias(\hat F_h)(y)})^2
\end{align*}
\end{proof}

\section{Asymptotic expansions, optimal bandwidth selection and stability}\label{sec:main_results}

Apart from the introduction of the DDEVD, another contribution of this work is the asymptotic expansion of the bias and variance in terms of $h$, which we present in the following theorem.
Throughout this section we assume that the base distribution \ac{cdf} $F_X$ is at least twice continuously differentiable.
Further, we require that the kernel $K$ is a~\ac{cdf} kernel, i.e. it is non-decreasing and has limits $\lim_{y \to -\infty} K(y) = 0$ and $\lim_{y \to \infty} K(y) = 1$.
The distribution described by this kernel, as well as by $K^2$, should at least have first and second moments.

\begin{theorem}[Asymptotic Expansion of the Mean Squared Error]\label{thm:asymptotic_expansion}
The Mean Squared Error, $\MSE(y)$, has the following asymptotic expansion up to order $\Ocal(h^3)$:
\begin{align*}
    \MSE(y) &= \frac{1}{m^2} \left[ \left(\sum_i b_{0,i}\right)^2 + \sum_i V_{0,i} \right] \\ 
            &+ \frac{1}{m^2} \left[ 2\left(\sum_i b_{0,i}\right)\left(\sum_i \bar{h_i} b_{1,i}\right) + \sum_i \bar{h_i} V_{1,i} \right] \\ 
            &+ \frac{1}{m^2} \left[ \left(\sum_i \bar{h_i} b_{1,i}\right)^2 + 2\left(\sum_i b_{0,i}\right)\left(\sum_i \bar{h_i^2} b_{2a,i} + \sum_i \bar{h_i}^2 b_{2b,i}\right) + \sum_i \left( \bar{h_i^2} V_{2a,i} + \bar{h_i}^2 V_{2b,i} \right) \right] \\
            &+ \mathcal{O}(h^3)
\end{align*}
This result stems from the individual expansions for the bias and variance, where for the $i$-th block:
\begin{align*}
    \Bias_i(y) &= b_{0,i}(y) + \bar{h_i} b_{1,i}(y) + \bar{h_i^2} b_{2a,i}(y) + \bar{h_i}^2 b_{2b,i}(y) +\mathcal{O}(h^3)\\
    \Var(F_{i,h}^{n_i})(y) &= V_{0,i}(y) + \bar{h_i} V_{1,i}(y) + \bar{h_i^2} V_{2a,i}(y) + \bar{h_i}^2 V_{2b,i}(y) + \mathcal{O}(h^3)
\end{align*}
The overall bias and variance are given as $\Bias(y) = \frac1m \sum_i \Bias_i(y)$ and $\Var(y) = \frac1{m^2}\sum_i \Var_i(y)$.
Exact expressions and large $n$ approximations for the coefficients $b_{s,i}(y)$ and $V_{s,i}(y)$ are derived in the Appendix.
\end{theorem}
\begin{proof}
See Appendix~\ref{appendix:asymptotic}.
\end{proof}

Note that the quality of such an expansion is dependent on the size of the bandwidth - smaller bandwidths yield a better approximation.
This can for example be achieved by scaling the data by a fixed factor to the range $[0,1]$ before applying the bandwidth selection procedure.
When rescaling by a factor of $s$ and computing the optimal bandwidth $h_{opt, s}$, the optimal bandwidth $h_{opt, s}$ can be expressed in terms of the original bandwidth $h_{opt}$ as follows:
\begin{align*}
    h_{opt, s} &= s h_{opt}
\end{align*}
This follows directly from the scaling property of the kernel density estimator and is a special case of more general transformations that we will discuss in Section~\ref{sec:transforms}.

By integrating the MSE over all $y$, we obtain the Mean Integrated Squared Error (MISE). We denote by $\mathbf{h} = (h_1, \ldots, h_m)^T$ a bandwidth vector that assigns a constant bandwidth to each block.

\begin{theorem}[Optimal Bandwidth]\label{thm:optimizer}
The bandwidth vector $\mathbf{h}_{opt}$ that minimizes the second-order approximation of the MISE (up to order $h^3$) is given by:
\begin{equation}\label{eq:h_opt}
    \mathbf{h}_{opt} = -\frac{1}{2}\mathbf{Q}^{-1}\mathbf{c}
\end{equation}
where $\mathbf{c} \in \R^m$ is the vector of coefficients for the linear terms in $\mathbf{h}$, and $\mathbf{Q} \in \R^{m \times m}$ is the Hessian matrix of the quadratic terms in $\mathbf{h}$ from the MISE expansion. The components of $\mathbf{c}$ and $\mathbf{Q}$ are given by:
\begin{align}
    c_i &= \int \left( 2\left(\sum_{j=1}^{m} b_{0,j}(y)\right)b_{1,i}(y) + V_{1,i}(y) \right) \mathrm{d}y \\
    Q_{kl} &= \int \left( b_{1,k}(y)b_{1,l}(y) + \delta_{kl} \left(2 \sum_{j=1}^{m}b_{0,j}(y) b_{2,k}(y) + V_{2,k}(y)\right) \right) \mathrm{d}y
\end{align}
where $b_{2,k} = b_{2a,k} + b_{2b,k}$ and $V_{2,k} = V_{2a,k} + V_{2b,k}$.
The optimal bandwidth vector obtained in this fashion is optimal also in a global sense, i.e. it also optimizes all individual bandwidths.
\end{theorem}
\begin{proof}
See Appendix~\ref{appendix:optimization}.
\end{proof}

Note that the boundaries of all occurring integrals depend on the considered MISE - if the full real line is considered, the integrals are over $\R$, if the $q$-MISE is considered, the integrals are over $[F_X^{-1}(q), \infty)$. For the proofs, we will consider the $q$-MISE for large $q$, as we are primarily interested in the tail behaviour of the estimator. To simplify notation, we will omit the boundary of the integrals in the following, but it should be kept in mind that they are not necessarily over the full real line.

\subsection{Practical Implementation via iterative Plug-in}

The optimal bandwidth vector $\mathbf{h}_{opt}$ derived in Theorem~\ref{thm:optimizer} depends on the functionals $\mathbf{c}$ and $\mathbf{Q}$.
These components involve integrals of the bias and variance coefficients, which in turn depend on the unknown base distribution $F_X$, its density $f_X$, and the derivative $f'_X$.
In practical applications, these quantities are not known a priori, rendering the direct calculation of Equation~\eqref{eq:h_opt} impossible.

To obtain a feasible estimator and reduce sensitivity to the initial pilot distribution, we employ an iterative plug-in approach.
This method treats the optimal bandwidth selection as a fixed-point problem where the bandwidth used to estimate the density functionals must be consistent with the bandwidth derived from optimizing the MISE.

\begin{enumerate}
    \item \textbf{Initialization:}
    We initialize the bandwidth vector $\mathbf{h}^{(0)}$ using a robust rule-of-thumb to ensure a reasonable starting scale. A standard choice is Silverman's rule~\cite{silverman2018density}:
    \begin{equation}
        h^{(0)}_{i} = 1.06 \cdot \hat{\sigma}_i \cdot n_i^{-1/5},
    \end{equation}
    where $\hat{\sigma}_i$ is the sample standard deviation of the $i$-th block.

    \item \textbf{Iterative Update:}
    For each iteration $k = 0, 1, \dots$, we perform the following steps:
    \begin{itemize}
        \item \textbf{Functional Estimation:} Using the current bandwidth estimate $\mathbf{h}^{(k)}$, we construct the pilot density $\hat{f}^{(k)}$ and pilot CDF $\hat{F}^{(k)}$:
        \begin{align}
            \hat{f}^{(k)}(y) &= \frac{1}{m} \sum_{i=1}^m \frac{1}{n_i} \sum_{j=1}^{n_i} \frac{1}{h^{(k)}_{i}} K'\left(\frac{y-X_{ij}}{h^{(k)}_{i}}\right) \\
            \hat{F}^{(k)}(y) &= \frac{1}{m} \sum_{i=1}^m \frac{1}{n_i} \sum_{j=1}^{n_i} K\left(\frac{y-X_{ij}}{h^{(k)}_{i}}\right)
        \end{align}
        These estimates are substituted into the definitions of $\alpha(y)$ and $\beta(y)$ to compute the empirical coefficients $\hat{\mathbf{c}}^{(k)}$ and Hessian matrix $\hat{\mathbf{Q}}^{(k)}$ via numerical integration.

        \item \textbf{Damped Optimization Step:} We calculate the raw optimal bandwidth update $\mathbf{h}^*$ by solving the system derived in Theorem~\ref{thm:optimizer}:
        \begin{equation}
            \mathbf{h}^* = -\frac{1}{2} (\hat{\mathbf{Q}}^{(k)})^{-1} \hat{\mathbf{c}}^{(k)}.
        \end{equation}
        To ensure numerical stability and prevent oscillation, we apply a damped update rule:
        \begin{equation}
            \mathbf{h}^{(k+1)} = (1 - \lambda)\mathbf{h}^{(k)} + \lambda \mathbf{h}^*,
        \end{equation}
        where $\lambda \in (0, 1]$ is a relaxation parameter (typically set to $\lambda \approx 0.5$).
    \end{itemize}

    \item \textbf{Termination:}
    The procedure is repeated until the relative change in bandwidth is small, satisfying $||\mathbf{h}^{(k+1)} - \mathbf{h}^{(k)}||_2 / ||\mathbf{h}^{(k)}||_2 < \epsilon$ for a chosen tolerance $\epsilon$.
\end{enumerate}

The damping parameter $\lambda$ provides a regularization in case the mapping from $\mathbf{h}^{(k)}$ to $\mathbf{h}^*$ is not a contraction.
We will investigate the stability of this iterative scheme numerically in Section~\ref{sec:numerical}.

Theoretical analysis of the stability conditions for this iterative plug-in method via fixed point theorems is a topic for future research.
Experiments again suggest a connection between the sample sizes $m$ and $n_i$.

\subsection{Stability of the Optimal Bandwidth Selection}
In order for the optimal solution $\mathbf{h}_{opt}$ to be stable, the Hessian matrix $\mathbf{Q}$ must be positive definite.

As can be seen in Figure~\ref{fig:stability_known}, this fact is by no means obvious and as we will see depends on a delicate balance between
terms related to the underlying distribution $F_X$, the chosen kernel $K$ as well as $m$ and $n_i$.

We will consider a special case where all $n_i$ are equal, i.e. $n_i=n$ for all $i$. This simplifies the analysis of the Hessian matrix while still providing valuable insights into the stability of the optimal solution. Another reasonable assumption will be that the mean of the kernel is zero, i.e. $\int x K'(x) \mathrm{d}x = 0$, which is true for most common kernels. This essentially gets rid of some implicit bias in the kernel density estimation.

\begin{remark}\label{remark:eigenvalues}
By setting $n_i=n$, the structure of the hessian matrix simplifies to a Toeplitz matrix with constant off diagonal elements, i.e. 
\[\mathbf{Q} = \mathbf{E} b + \mathbf{I} a,\]
where $\mathbf{E}$ is the all ones matrix and $\mathbf{I}$ is the identity matrix. The diagonal elements are $a+b$ and the off diagonal elements are $b$. 

The eigenvalue structure of this matrix is also especially simple. It can be found by observing that 
\[ \mathbf{Q}\mathbf{x} = b \sum_i x_i \mathbf{1} + a \mathbf{x}\]
with $\mathbf{1} = (1,1,\ldots,1)^T$. For $\sum_i x_i = 0$, we have $\mathbf{Q}\mathbf{x} = a \mathbf{x}$, which means that $a$ is an eigenvalue with multiplicity $m-1$. 
For $\mathbf{x} = \mathbf{1}$, we have $\mathbf{Q}\mathbf{1} = a + m b \mathbf{1}$, which means that $a+mb$ is an eigenvalue with multiplicity $1$.
\end{remark}

\begin{lemma}\label{lemma:positive_definite}
Let $n_i=n>1$ for all $i\in[m]$. Then the matrix $\mathbf{Q}$ is positive definite if and only if \[D := \int 2 m b_{0}(y) b_{2}(y) + V_{2}(y)\mathrm{d}y > 0\]
\end{lemma}
Note that we omitted the subscripts $k$ and $l$ in the above expression for $D$ since
all the functions $b_{0,k}$, $b_{2,k}$, and $V_{2,k}$ are the same for all $k$.
\begin{proof}
The matrix $\mathbf{Q}$ is positive definite if and only if all its eigenvalues are positive.
From the previous remark, we know that the eigenvalues of $\mathbf{Q}$ are $a$ and $a+mb$.
Since $b = \int b_{1}(y)^2 \mathrm{d}y \geq 0$, we know that $mb \geq 0$, thus $a+mb \geq a$,
so it suffices to show that the smallest eigenvalue $a > 0$. Since $a \propto D$, this is the case if and only if $D > 0$.
\end{proof}

The stability of the optimal bandwidth solution is determined as follows:

\begin{theorem}[Asymptotic Stability of the Optimal Bandwidth]\label{theorem:stability}
Let $n_i = n > 1$ for all $i \in \{1, ..., m\}$, assume a kernel $K$ with zero mean ($\mu_{K,1} = \int uK'(u)du = 0$), and consider the asymptotic regime where $n \to \infty$.
The positive definiteness of the Hessian matrix $Q$, required for a stable \ac{mise} minimum, depends on the tail behavior of the base distribution $F_X$ categorized by its extreme value index $\gamma$ with $\gamma > -\frac12$.
For the optimal bandwidth solution to be stable, we need
\begin{align*}
    m < C(\gamma, F_X, K) \cdot n^{1+\frac{\gamma}{2}},
\end{align*}
where $C(\gamma, F_X, K)$ is a positive constant that depends on the extreme value index $\gamma$, the base distribution $F_X$, and the kernel $K$ as derived in Appendix~\ref{appendix:stability}.
The optimization is therefore asymptotically stable, provided the number of blocks does not grow faster than $\mathcal{O}(n^{1+\frac{\gamma}{2}})$.
\end{theorem}

\begin{proof}
The proof follows from Lemma~\ref{lemma:positive_definite}, which establishes that the positive definiteness of $\mathbf{Q}$ is equivalent to the condition $D > 0$.
The asymptotic analysis in Appendix~\ref{appendix:stability} derives the leading-order behavior of the integral $D$ for large $n$.
\end{proof}

This is a significant result. It shows that the optimization success hinges critically on the tail behaviour of the underlying distribution.
The Gumbel case is only partially covered by the above analysis, because of the special asymptotic representation of the distributions
used in the proof. The result will, however, be on the same order up to logarithmic corrections.

Up to this point, the case where all $n_i$ were equal has been considered. However, in practical applications, this is often not the case.
The result above critically hinged on the assumption of equal block sizes, due to the Toeplitz structure of the matrix $Q$ that arises in that case.
For the more general setting, perturbation theory can be employed to analyze the impact of unequal block sizes on the optimization procedure.
This can be seen by setting $\bar{n} = \frac{1}{m} \sum_i n_i$ and observing the identity
\begin{align*}
n_i &= \bar{n} + (n_i - \bar{n})\\
    &= \bar{n} \left(1+\frac{n_i-\bar{n}}{\bar{n}}\right)\\
    &=: \bar{n} (1+\epsilon_i),
\end{align*}
where $\epsilon_i$ is a small perturbation if the $n_i$ do not deviate too much from their mean.
However, the details of this analysis are beyond the scope of the current work and will be addressed in future research,
we only sketch a brief outline of a possible result.

\begin{conjecture}[Conjectured Stability for Unequal Block Sizes]\label{conjecture:unequal_block_sizes}
Let the conditions of Theorem~\ref{theorem:stability} hold, but now for block sizes $n_i$ that are not necessarily equal. Let $\bar{n} = \frac{1}{m}\sum_i n_i$ be the mean block size and $\sigma_n^2 = \frac{1}{m}\sum_i (n_i - \bar{n})^2$ be the variance of the block sizes.
If the relative standard deviation of the block sizes is small, i.e., $\sigma_n / \bar{n} \ll 1$, then the stability condition of Theorem~\ref{theorem:stability} holds approximately with $n$ replaced by $\bar{n}$.
Specifically, if $m < C(\gamma, F_X, K) \cdot \bar{n}^{1+\frac{\gamma}{2}}$ holds with a sufficient margin, the Hessian matrix $\mathbf{Q}$ remains positive definite.
\end{conjecture}
\begin{proof}[Proof Sketch]
The Hessian matrix $\mathbf{Q}$ for unequal block sizes $n_i$ can be seen as a perturbation of the idealized Hessian $\mathbf{Q}_0$ corresponding to equal block sizes $\bar{n}$.
We can write $\mathbf{Q} = \mathbf{Q}_0 + \delta\mathbf{Q}$, where $\mathbf{Q}_0$ is the matrix from the proof of Lemma~\ref{lemma:positive_definite} evaluated at $n=\bar{n}$.
The validity of this decomposition can for example be seen by performing a Taylor series expansion of the $b$ and $V$ terms in the neighborhood of $\bar{n}$.
The exact computations are very lengthy and will not be carried out here.
However, the entries of the perturbation matrix $\delta\mathbf{Q}$ depend to leading order on the deviations $\epsilon_i = (n_i - \bar{n})/\bar{n}$.
For small $\epsilon_i$, the entries of $\delta\mathbf{Q}$ are approximately linear in $\epsilon_i$.

From the theory of matrix perturbations (e.g., Weyl's inequality), the eigenvalues of $\mathbf{Q}$ are close to the eigenvalues of $\mathbf{Q}_0$.
Let $\lambda_{\min}(\mathbf{Q})$ denote the smallest eigenvalue of the matrix $\mathbf{Q}$. Then we have
\[ \lambda_{\min}(\mathbf{Q}) \ge \lambda_{\min}(\mathbf{Q}_0) - ||\delta\mathbf{Q}||_2, \]
where $||\cdot||_2$ is the spectral norm.
From the analysis for equal block sizes, we know that $\lambda_{\min}(\mathbf{Q}_0) \propto D(\bar n)$, where $D(\bar n) > 0$ is the stability condition derived in Lemma~\ref{lemma:positive_definite} with $n=\bar{n}$.
The norm of the perturbation matrix can be bounded by the magnitude of the deviations, $||\delta\mathbf{Q}||_2 \le K \cdot \max_i|\epsilon_i|$ for some constant $K$ that depends on the derivatives of the matrix entries with respect to $n$.
Therefore, $\mathbf{Q}$ remains positive definite if $D(\bar{n}) > ||\delta\mathbf{Q}||_2$.
This condition is satisfied if the stability condition for the mean block size $\bar{n}$ holds with a margin large enough to absorb the perturbation, which is guaranteed if the variance of the block sizes is sufficiently small.
\end{proof}

%
%

\section{Data transformations}\label{sec:transforms}

\subsection{When kernel methods need help}

The bandwidth selector of Theorem~\ref{thm:optimizer} is optimal in the sense of minimizing the MISE, but this does not necessarily mean that the resulting estimator will perform well in practice.  In particular, the optimal bandwidth might be very small, leading to an estimator that is essentially a staircase in the tail and thus fails to capture the underlying distribution's true tail behaviour.  This can happen when the underlying distribution has for example a sharp spike near zero, which is common in many real-world distributions. If still a heavy tail is present, the optimal bandwidth will be small to capture the sharp spike, but this will lead to a poor fit in the tail. Resorting to a larger quantile $q$ for the $q$-MISE might solve this issue, but is typically not feasible due to data limitations.
The standard remedy in the kernel density estimation literature is to fit on a transformed scale where the density is better conditioned, and back-transform the result.
This idea is developed in detail by Wand, Marron, and Ruppert \cite{wand1991transformations}, who propose a shifted-power (Box--Cox) family of transformations and an empirical procedure for selecting the transformation parameter.
Silverman's monograph \cite{silverman2018density} discusses the same idea under the heading of ``variable bandwidth methods''.
Sheather's review \cite{sheather2004density} and Park and Marron \cite{park1990comparison} cover the bandwidth-selection complications that accompany transformation-based estimation.
We adopt the same principle for DDEVD.
Note that this section is much more qualitative than the previous sections, as its main motivation is to justify the use of transformations in practice, rather than to provide a rigorous theoretical analysis of their properties. The latter is an interesting topic for future research, but is again beyond the scope of the current work.

\subsection{The transformation framework}

Let $\mathcal{Y}\subset\R$ and $\mathcal{Z}\subset\R$ be two subsets of the real line, and let $T: \mathcal{Y} \to \mathcal{Z}$ be a strictly increasing $\mathcal{C}^2(\mathcal{Y}, \mathcal{Z})$ map and $Z = T(Y)$ be the transformed variable.
Let $F_Y$ and $F_Z$ be the CDFs of $Y$ and $Z$, respectively; by the change-of-variables identity, $F_Y(y) = F_Z(T(y))$.
Further, we denote by $\hat F_Z$ the DDEVD estimator fitted on the transformed $Z$-data.

\begin{definition}[Transformation-based DDEVD]
The \emph{back-transformed DDEVD estimator} $\hat F_Y^T$ is defined as the composition of the DDEVD estimator fitted on the transformed data with the transformation $T$:
\begin{equation}\label{eq:transform_cdf}
\hat F_Y^T(y) \;:=\; (\hat F_Z \circ T)(y) \;=\; \hat F_Z(T(y)).
\end{equation}
\end{definition}

The back-transformed estimator is a CDF estimator on $\mathcal{Y}$ targeting the same target distribution $F_Y$ as the direct fit $\hat F_Y^{\mathrm{id}}$ (the case $T = \mathrm{id}$), preserving the DDEVD structure of Section~\ref{sec:main_results}.
Concretely, fitting DDEVD on $z_{ij} = T(y_{ij})$ produces the $m$-average of $n_i$-th powers of block kernel CDFs on the $z$-data, so evaluating this estimator at $T(y)$ delivers $\hat F_Y^T(y)$.
The qualitative content of the transformation, established in the KDE literature, is that the back-transformed estimator behaves like a y-scale kernel CDF with a \emph{locally varying} bandwidth $h_Y(y) \propto 1/|T'(y)|$.
For the log transform on positive support, $T'(y) = 1/y$ gives $h_Y(y) \propto y$ resulting in a narrow bandwidth near zero (where the data is dense) and a wide bandwidth in the tail (where it is sparse).
This is precisely the local adaptation that the single-bandwidth selector cannot supply.
Also note, that in case of an affine transformation $T(y) = s y$ with $s > 0$, the back-transformed estimator is just a rescaled version of the original estimator, i.e., $\hat F_Y^T(y) = \hat F_Z(s y)$, and the optimal bandwidth vector $\mathbf{h}_{opt}$ for the transformed data is just a rescaled version of the optimal bandwidth vector for the original data, i.e., $h_{opt, s} = s h_{opt}$, as has alreadyh been mentioned in Section~\ref{sec:main_results}.

\subsection{Why transformation helps: a qualitative argument}

We borrow the qualitative answer to ``when does transformation help?'' from the kernel-density estimation literature, with the caveat that the corresponding DDEVD analysis is more involved (see the remark at the end of this subsection). Following \cite{wand1991transformations}, the AMISE of a kernel density estimator scales with the integrated squared density curvature $R(f) := \int (f'')^2$; smaller $R(f)$ means easier estimation.
Transformation helps if $R(f_Z) < R(f_Y)$, i.e.\ if the transformed density is smoother in this sense than the original.
For kernel CDF estimation the analogous AMISE roughness functional is $\int (f')^2$ rather than $\int (f'')^2$ \cite{azzalini1981note, polansky2000multistage}; the qualitative picture below is unchanged, only the precise functional that the transformation reduces.
A sharp bulk near zero alone is sufficient to make $R(f_Y)$ very large (or, in pathological cases, unbounded), and any transformation that maps the bulk to a moderate-density region of the transformed scale dramatically reduces $R(f_Z)$.
A heavy upper tail compounds the problem along a separate axis: the global bandwidth selector is forced toward the bulk's feature scale, leaving the sparse tail under-smoothed.
This is the regime where transformation provides a genuine benefit; the principle is independent of the specific kernel-CDF / DDEVD setting.

When $f_Y$ has a near-singular density at $y \to 0^+$, the no-transform fit fails the standard diagnostic checks (Section~\ref{subsec:diagnostics}), and the canonical log transform maps the singular region to $z \to -\infty$ on a well-behaved log-normal-like density.
The transformation rescues the estimator's behaviour in the tail.
Symmetric thin-tailed data without boundary structure does not benefit from transformation \cite{wand1991transformations, sheather2004density}; the recipe below applies only when the diagnostic checks fail on the original scale.

The computation of the AMISE for the DDEVD estimator is more involved than for the KDE, and the transformation's effect on the DDEVD MISE is not as straightforward as in the KDE case.
A rigorous analysis of the transformation's effect on the DDEVD MISE and the AMISE is an interesting topic for future research, but is yet again beyond the scope of the current work.

\subsection{Diagnostics: when to reach for the transformation}
\label{subsec:diagnostics}

We do not know $f_Y$ a priori, the case for transformation is made empirically through three observable diagnostic checks on the y-scale fit.

\paragraph{Iteration-convergence diagnostic.}
The iterative plug-in either converges to a finite bandwidth or fails (negative bandwidths, oscillation, $\mathbf{h}^{(k)} \to 0$).
Failure indicates the Hessian $\mathbf{Q}$ is not positive-definite at the converged point.

\paragraph{Bandwidth-vs-data-spacing diagnostic.}
The converged $h_Y^\star$ should be at least the order of the median spacing between consecutive q-tail observations.
A bandwidth below this threshold cannot smooth between adjacent tail points.
The resulting kernel CDF will be a staircase in the q-tail.

\paragraph{Density-spikiness diagnostic.}
The estimated density $\hat f_X = \hat F_X'$ should be smooth in the q-tail.
Visible multimodality with one local maximum per observation indicates the bandwidth is below the local feature scale.

When these diagnostics pass, the no-transform fit is accepted.
When one or more fail, the transformation remedy is justified.
The same diagnostics applied on the transformed scale tell us whether the transformation worked.

\subsection{Practical recipe}\label{subsec:recipe}

For DDEVD with the iterative plug-in:

\begin{enumerate}
\item Fit on the y-scale data using the standard plug-in.
\item Apply the three diagnostic checks of Section~\ref{subsec:diagnostics}.
  If all pass, accept the no-transform fit.
\item If any diagnostic fails, identify a transformation $T$ from the data's boundary structure. Typical candidates include $\log$ or Box-Cox.
\item Refit on the transformed data, apply the same diagnostics on the transformed fit, and accept the back-transformed estimator $\hat F_Y^T = \hat F_Z \circ T$ when the transformed diagnostics pass.
\item Return levels are computed on the working scale and back-transformed: $\hat z_T = \hat F_Z^{-1}(1 - 1/T_{\rm ret})$, $\hat y_{T_{\rm ret}} = T^{-1}(\hat z_T)$.
\end{enumerate}

\subsection{Outlook}\label{subsec:outlook}

The qualitative argument above leans on the kernel-density-estimation literature for two propositions: that transformation helps when the working-scale density has lower roughness $R(f_Z) < R(f_Y)$, and that this rough inequality is decidable from data through post-fit diagnostics.
Both are operational rather than theoretical statements in our setting.
Making them quantitative for the specific DDEVD $q$-MISE expansion of Theorem~\ref{thm:asymptotic_expansion} is a worthwhile follow-up.
Three concrete directions:

The leading-order q-MISE expansion of Theorem~\ref{thm:asymptotic_expansion} can be evaluated on the transformed scale, yielding a Jacobian-weighted analogue of the y-scale constant.
Comparing the two constants for canonical $(F_X, T)$ pairs would give the DDEVD analogue of \cite{wand1991transformations}'s analytical comparison for KDE.

A data-driven transformation selector within the Box--Cox family, analogous to the KDE selectors of \cite{wand1991transformations, park1990comparison}, would replace the diagnostic-and-default workflow of Section~\ref{subsec:recipe} with an optimisation of the transformation parameter against an estimated q-MISE constant.

The conditioning failure modes diagnosed in Section~\ref{subsec:diagnostics} are likely traceable to boundary-regularity conditions on $f_X$ that the proof of Theorem~\ref{theorem:stability} implicitly assumes.
Closing this gap would give a formal characterisation of which $F_X$ require transformation and which do not, complementing the empirical diagnostics with a theoretical prediction.

These directions extend the present empirical justification into a full theoretical treatment.
For the scope of this paper, the diagnostic-and-default workflow combined with the KDE-literature precedent is sufficient to justify the transformation remedy in practice, an applied case study (sub-daily rainfall) is
deferred to a companion paper.

\section{Numerical Experiments}\label{sec:numerical}
To validate the theoretical findings, we conducted numerical experiments to assess the performance of the \ac{ddevd} estimator and the stability of the optimal bandwidth selection procedure.\footnote{The code to reproduce all experiments is available at \url{https://github.com/DataLabHell/ddevd}. This repository also includes a package that implements the DDEVD estimator and the optimal bandwidth selection procedure, which can be used for further experimentation and application.}

\subsection{Synthetic Data Experiments}

By assuming that the base distribution $F_X$ is known, we can generate synthetic datasets to evaluate the estimator's performance under varying conditions.
On one hand we can control the tail behaviour of the base distribution by selecting different distributions with known extreme value indices $\gamma$.
On the other hand, we can vary the number of blocks $m$ and the number of samples per block $n_i$ to observe their influence on the stability of the bandwidth optimization.

In addition, by pretending we do not know $F_X$, we can check the accuracy of the plug-in estimates for the unknown quantities in the optimal bandwidth formula.

\subsubsection{Known Base Distribution - Stability Validation}

This is arguably the simplest setting to validate the theoretical results. We generate $m$ blocks of $n$ i.i.d. samples from a known base distribution $F_X$.
For four different base distributions, we vary $m$ and $n$ systematically and check whether the bandwidth optimization procedure converges stably.
We consider the following base distributions:
\begin{itemize}
    \item Standard Normal Distribution ($\gamma = 0$)
    \item Exponential Distribution ($\gamma = 0$)
    \item Pareto Distribution with shape parameter $\alpha = 2$ ($\gamma = 0.5$)
    \item Cauchy Distribution ($\gamma = 1$)
\end{itemize}
For each combination of $m$ and $n$, we run the bandwidth optimization procedure and record whether it converges to a stable solution.
The results are summarized in Figure~\ref{fig:stability_known}.

\begin{figure}[ht]
    \centering
    \includegraphics[width=0.85\textwidth]{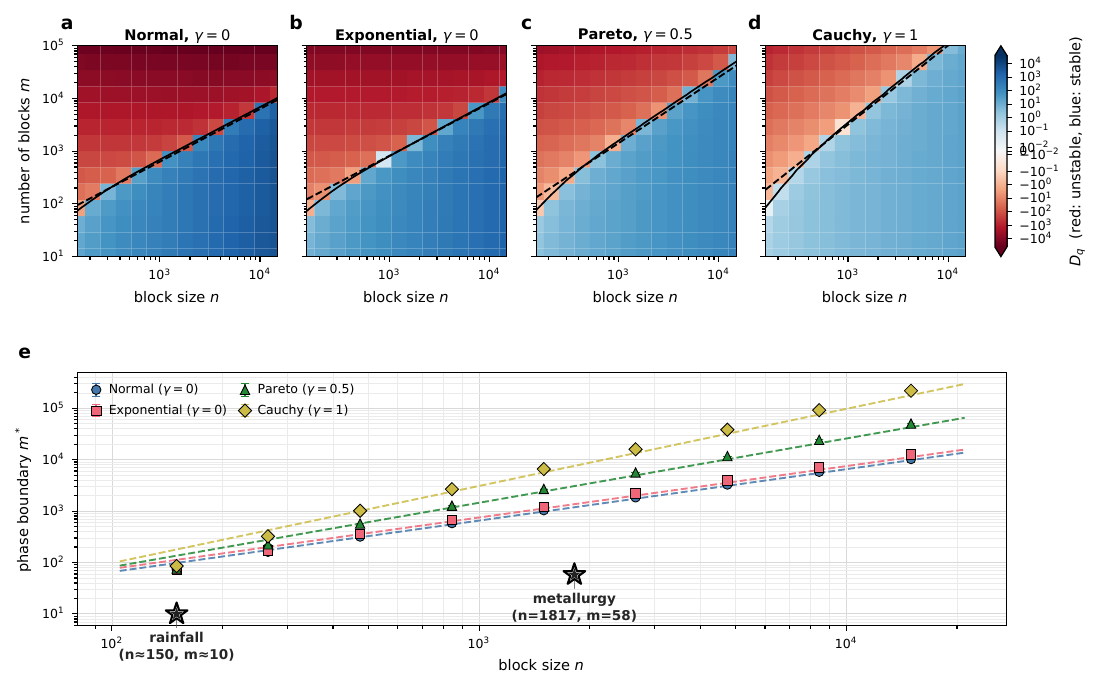}
    \caption{
        Stability phase diagram for the DDEVD bandwidth optimisation for different distributions.
        Each point in the $(n, m)$ plane is coloured according to whether the bandwidth optimisation converges stably (blue) or diverges (red), evaluated numerically for a representative base distribution.
        The theoretical stability boundary $m = C \cdot n^{1+\gamma/2}$ (dashed line) separates the two regimes.
        The figure parts a, b, c and d show the Normal, Exponential, Pareto and Cauchy distribution, respectively.
        Subfigure e shows all evaluated distributions and their transition points together with the theoretical stability boundaries.
    }
    \label{fig:stability_known}
\end{figure}

\subsubsection{Unknown Base Distribution - Stability with Plug-in Estimator Evaluation}

In a more realistic scenario, we assume that the base distribution $F_X$ is unknown.
We generate synthetic datasets as before, but this time we estimate the unknown quantities in the optimal bandwidth formula using plug-in estimators based on the observed data.
We then compute the \ac{ddevd} estimator with the plug-in optimal bandwidth and evaluate its \ac{mise} as before.
The stability of the bandwidth optimization is in a similar order as in the known base distribution case (see Figure~\ref{fig:stability_unknown}), albeit with a slightly higher variance due to the estimation of the unknown quantities.
Further, the experiments need to be done at lower values of $m$ and $n$ due to the increased computational cost of the plug-in estimation procedure.
The \ac{mise} of the plug-in DDEVD estimator is compared against the \ac{mise} obtained with the analytically optimal bandwidth in Figure~\ref{fig:mise_unknown}.

\begin{figure}[ht]
    \centering
    \includegraphics[width=0.85\textwidth]{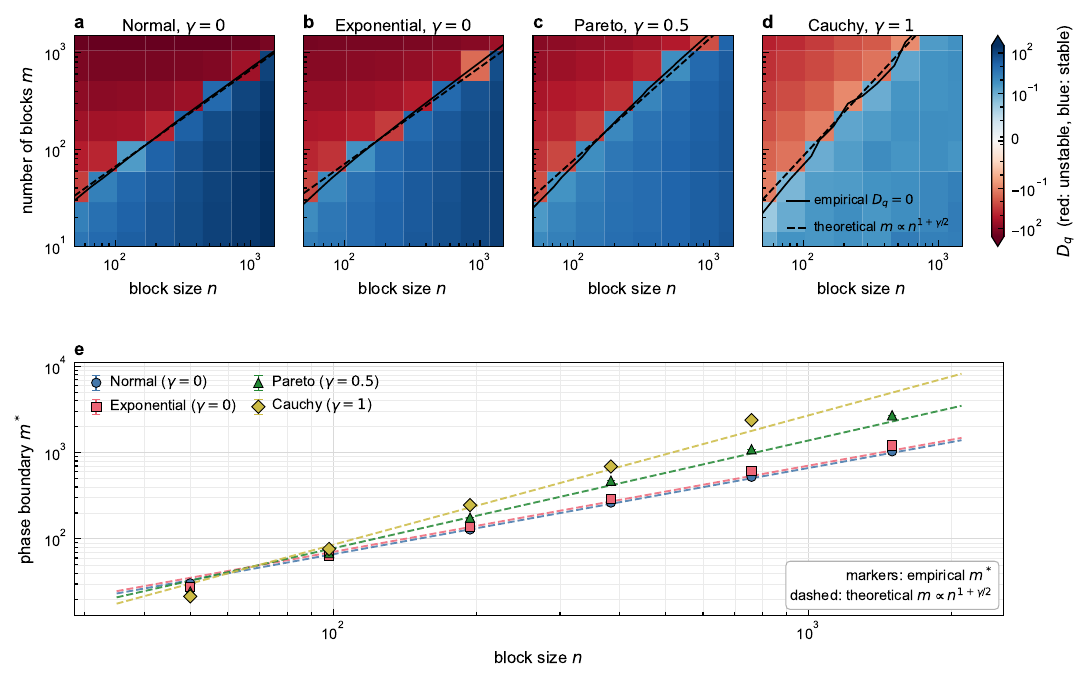}
    \caption{
        Stability phase diagram for the DDEVD bandwidth optimisation for different distributions with unknown base distribution.
        Each point in the $(n, m)$ plane is coloured according to whether the bandwidth optimisation converges stably (blue) or diverges (red), evaluated numerically for a representative base distribution.
        The theoretical stability boundary $m = C \cdot n^{1+\gamma/2}$ (dashed line) separates the two regimes.
        The figure parts a, b, c and d show the Normal, Exponential, Pareto and Cauchy distribution, respectively.
        Subfigure e shows all evaluated distributions and their transition points together with the theoretical stability boundaries.
    }
    \label{fig:stability_unknown}
\end{figure}

\begin{figure}[ht]
    \centering
    \includegraphics[width=0.85\textwidth]{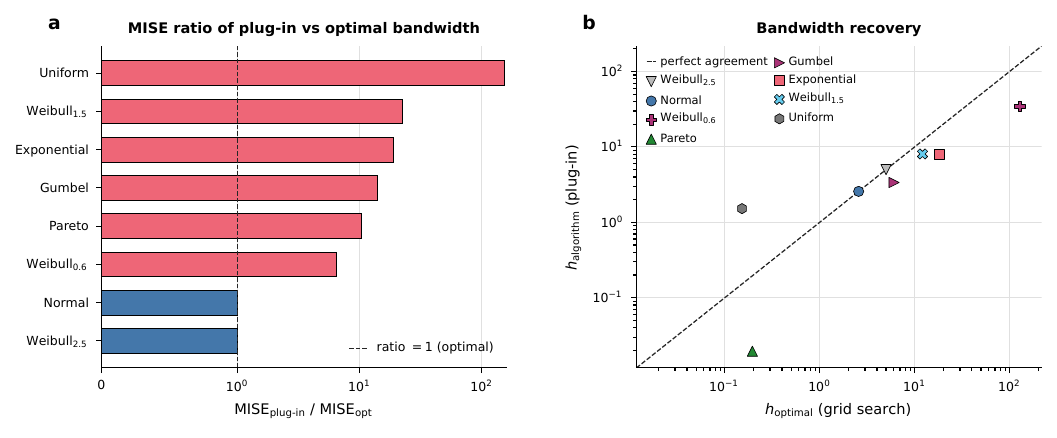}
    \caption{
        \ac{mise} of the DDEVD estimator when the base distribution $F_X$ is unknown and its functionals are replaced by plug-in estimates inside the bandwidth optimiser.
        The achieved \ac{mise} is contrasted with the \ac{mise} of the same estimator using the analytically optimal bandwidth, quantifying the cost of not knowing $F_X$ a priori.
    }
    \label{fig:mise_unknown}
\end{figure}

\subsection{Varying Block Sizes - Stability Evaluation}

In practical applications, the number of samples per block $n_i$ may vary across blocks.
As conjectured in Conjecture~\ref{conjecture:unequal_block_sizes}, we investigate the stability of the bandwidth optimization procedure under varying block sizes.
We again generate synthetic datasets, this time with varying $n_i$, and assess the stability of the bandwidth selection.
The results are summarized in Figure~\ref{fig:stability_varying}.

\begin{figure}[ht]
    \centering
    \includegraphics[width=0.85\textwidth]{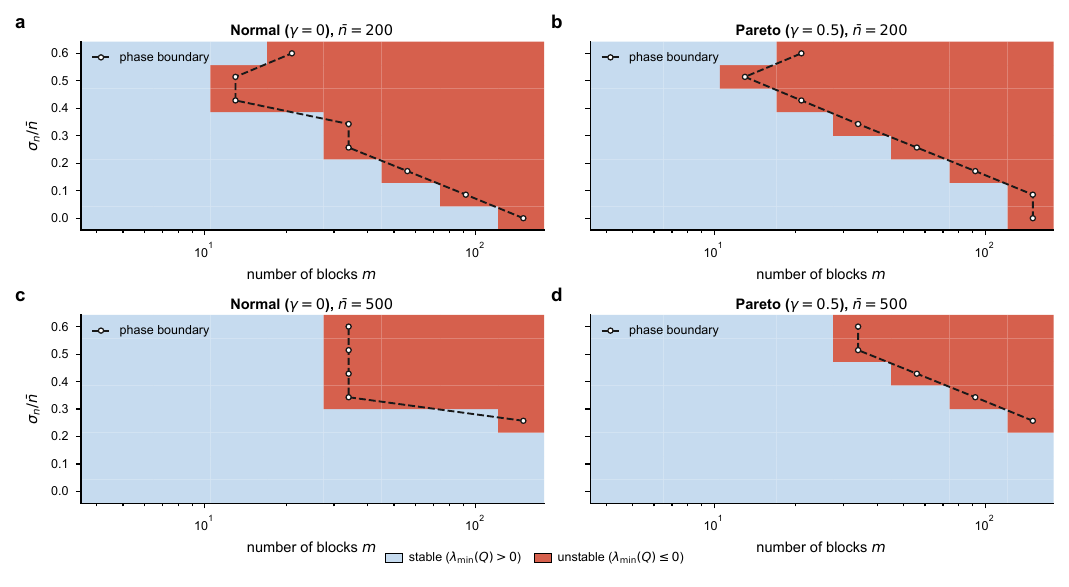}
    \caption{
        Stability of the DDEVD bandwidth optimisation under varying block sizes.
        With a fixed mean block size $\bar{n}$, the relative standard deviation $\sigma_n / \bar{n}$ of the block sizes is increased and the convergence of the bandwidth optimisation is recorded.
        The results confirm the conjectured behaviour of Conjecture~\ref{conjecture:unequal_block_sizes}: stability is preserved provided the variation in block sizes remains moderate.
    }
    \label{fig:stability_varying}
\end{figure}

\subsection{Known Base Distribution - MISE Evaluation}

To evaluate the \ac{mise} of the \ac{ddevd} estimator with the optimal bandwidth, we again generate synthetic datasets from known base distributions.
For each dataset, we compute the \ac{ddevd} estimator using the optimal bandwidth derived in Section~\ref{sec:main_results}.
We then estimate the \ac{mise} by averaging the integrated squared error over multiple runs.
The results are presented in Figure~\ref{fig:mise_known}.

\begin{figure}[ht]
    \centering
    \includegraphics[width=0.85\textwidth]{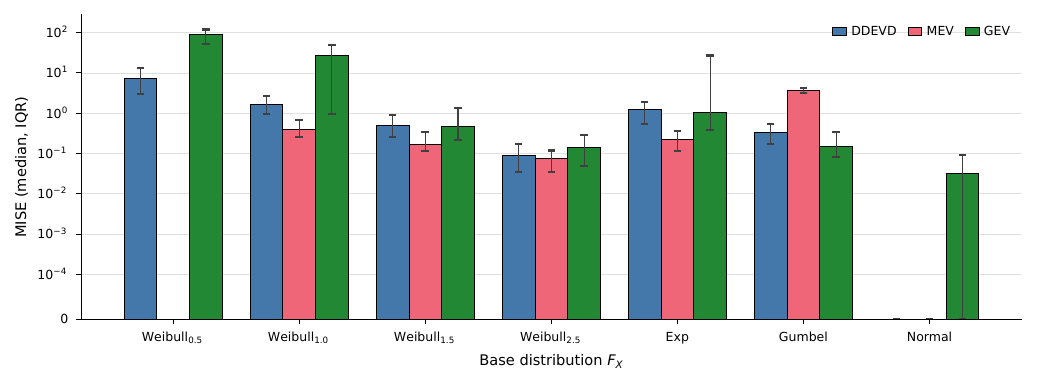}
    \caption{
        \ac{mise} of the DDEVD estimator with the analytically derived optimal bandwidth across a range of base distributions with known extreme value index $\gamma$.
        For each distribution, the integrated squared error is averaged over multiple synthetic samples.
        The DDEVD estimator is compared against benchmark estimators (\ac{gev} and \ac{mev}) to put the achieved error magnitudes into context.
    }
    \label{fig:mise_known}
\end{figure}

\section{Conclusion}\label{sec:conclusion}

In this work, we introduced the \ac{ddevd} estimator, a new non-parametric framework for estimating the distribution of extreme events.
Moving beyond the mere definition of the estimator itself, our primary contribution has been the rigorous theoretical analysis of its performance in the asymptotic regime $n\to\infty$.

Our analysis of the estimator's bias and variance has led to the derivation of an asymptotic expansion for the \ac{mise} of the estimator.
Using this expansion we subsequently solved for the explicit optimal bandwidth vector that minimizes this error criterion.

We made several key observations during this process. Notably, the analysis showed that the optimal bandwidth vector $h_{opt}$ is constant for each block of the data.
It can be computed by using an explicit formula, thereby getting rid of computationally intensive methods like cross-validation. Perhaps more importantly, we were
able to derive stability conditions under which the estimator maintains its performance.
It was shown that this is intrinsically linked to the extreme value index $\gamma$ of the underlying base distribution $F_X$. We have shown that for distributions
with $\gamma > -\frac12$, the stated optimization procedure is asymptotically stable provided the number of blocks
$m$ does not grow faster than $\mathcal{O}(n^{1+\frac{\gamma}{2}})$.

\subsection{Open problems and future work}\label{subsec:future_work}

The present paper proves several core results for the \ac{ddevd} estimator -- the asymptotic \ac{mise} expansion (Theorem~\ref{thm:asymptotic_expansion}), the closed-form optimal bandwidth (Theorem~\ref{thm:optimizer}), and the stability condition $m < C \cdot n^{1+\gamma/2}$ for equal block sizes and $\gamma > -\tfrac12$ (Theorem~\ref{theorem:stability}) -- but a number of gaps were flagged along the way and are collected here for clarity.

\paragraph{Stability theory.}
\begin{itemize}
    \item \emph{Iterative plug-in convergence.} The damped fixed-point scheme of Section~\ref{sec:main_results} is justified empirically; a rigorous convergence theory via contraction-mapping or fixed-point arguments, and an analytical characterisation of the relation between $\lambda$, $m$, and $n$ that guarantees a contraction, is open.
    \item \emph{Unequal block sizes.} Conjecture~\ref{conjecture:unequal_block_sizes} gives a perturbation-theoretic sketch of stability for unequal $n_i$, but the full Taylor expansion of the $b$- and $V$-coefficients around $\bar{n}$ and the resulting sharp bound on $\|\delta\mathbf{Q}\|$ remain to be carried out.
    \item \emph{Gumbel domain.} The stability proof covers $\gamma > -\tfrac12$ to leading order; the $\gamma = 0$ case is only partially handled because of the special asymptotic representation used in the proof. Closing this case rigorously requires a Watson-lemma modification adapted to densities of the form $f(t) \sim D\, t \log^\beta t$ near the upper endpoint.
    \item \emph{Boundary regularity of $f_X$.} The conditioning failure modes diagnosed in Section~\ref{subsec:diagnostics} are likely traceable to boundary-regularity assumptions on $f_X$ that the proof of Theorem~\ref{theorem:stability} makes implicitly. A formal characterisation of which $F_X$ require transformation, complementing the empirical diagnostics, is open.
\end{itemize}

\paragraph{Beyond the \ac{mise} criterion.}
The growth constraint $m = \Ocal(n^{1+\gamma/2})$ is a property of the \ac{mise} criterion, not of the estimator itself. Replacing \ac{mise} with an alternative target -- e.g.\ a tail-weighted loss, a quantile-loss, or a divergence-based criterion -- may relax the constraint and yield bandwidth selectors with different stability profiles.

\paragraph{Transformation theory.}
Section~\ref{sec:transforms} borrows the qualitative ``transformation helps if $R(f_Z) < R(f_Y)$'' heuristic from the KDE literature. Three concrete follow-ups remain:
\begin{itemize}
    \item Evaluate the leading-order $q$-\ac{mise} expansion of Theorem~\ref{thm:asymptotic_expansion} on the transformed scale, yielding a Jacobian-weighted analogue of the $y$-scale constant. Comparing canonical $(F_X, T)$ pairs would give the DDEVD analogue of \cite{wand1991transformations}.
    \item Build a data-driven transformation selector within the Box--Cox family by optimising an estimated $q$-\ac{mise} constant over the transformation parameter, in the spirit of \cite{wand1991transformations, park1990comparison}.
    \item Identify the precise roughness functional analogous to $\int (f')^2$ that drives DDEVD bias under transformation, given the $n_i$-th-power structure of the estimator.
\end{itemize}

\paragraph{Applications.}
Application to real-world settings where parametric \ac{evt} methods are prone to failure is a crucial next step.
Applied case studies on sub-daily rainfall, grain-size statistics in materials science, and other domains with a sharp boundary and heavy tail are deferred to a companion paper.

\subsubsection*{Acknowledgments}
{
  We thank the Data Lab Hell GmbH for supporting this research.
  We also thank our colleagues for their valuable feedback and discussions that contributed to the development of this work.
  Google Gemini was used to assist in the writing and editing of this manuscript.
  Claude Code was used to assist in the code implementation and numerical experiments.
}

\appendix
\section{Proof of Theorem~\ref{thm:asymptotic_expansion}}\label{appendix:asymptotic}

The proof requires us to derive asymptotic expansions for $\Bias_i(y)$ and $\Var(F_{i,h}^{n_i})(y)$ in powers of $h$.

\subsection{MSE from Bias and Variance expansions}

We have $\Bias(y) =  \E[\hat F_h(y)] - F(y) = \frac1m \sum_i \E F_{i, h}^{n_i}(y) - F_X^{n_i}$, so by defining
$\Bias_i(y) = \E F_{i,h}^{n_i} - F_X^{n_i}$ we can perform an expansion in terms of $h_{ij}$ up to third order. This
yield $\Bias_i(y) = b_{0,i}(y) + \bar{h_i} b_{1,i}(y) + \bar{h_i^2} b_{2a,i}(y) + \bar{h_i}^2 b_{2b,i}(y) +\mathcal{O}(h^3)$.
In a similar fashion we can derive an asymptotic expansion for 
$\Var(F_{i,h}^{n_i})(y) = V_{0,i}(y) + \bar{h_i} V_{1,i}(y) + \bar{h_i^2} V_{2a,i}(y) + \bar{h_i}^2 V_{2b,i}(y) + \mathcal{O}(h^3)$,
the exact and asymptotic forms of the coefficients will be derived in the following section. For now we will use this representation
to find an expansion of $\MSE$ in terms of $h$.

Note that by virtue of Theorem~\ref{thm:bias_variance} we can use bias and variance to express the $MSE$. For this, we need
\begin{align*}
\Bias^2(y) &= \left(\frac1m \sum_i \Bias_i(y)\right)^2 \\
           &= \left(\frac1m \sum_i b_{0,i}(y) + \bar{h_i} b_{1,i}(y) + \bar{h_i^2} b_{2a,i}(y) + \bar{h_i}^2 b_{2b,i}(y) +\mathcal{O}(h^3)\right)^2.
\end{align*}
By squaring this term and dropping higher order terms, we arrive at
\begin{align*}
\Bias^2(y) &= \frac{1}{m^2} \Bigg[\left(\sum_i b_{0,i}\right)^2 \\
           &+ 2\left(\sum_i b_{0,i}\right)\left(\sum_i \bar{h_i} b_{1,i}\right) \\
           &+ \left(\sum_i \bar{h_i} b_{1,i}\right)^2 + 2\left(\sum_i b_{0,i}\right)\left(\sum_i \bar{h_i^2} b_{2a,i} + \sum_i \bar{h_i}^2 b_{2b,i}\right) \Bigg] \\
           &+ \mathcal{O}(h^3)
\end{align*}

By adding this with the expanded Variance term, the $\MSE$ stated in Theorem~\ref{thm:asymptotic_expansion} can be obtained by simple addition.

\subsection{Explicit Forms of Coefficients}

The proofs largely rely on the binomial expansion. We shortly recall the binomial theorem:
\begin{theorem}[Binomial Theorem]
For any $x,y \in \mathbb{R}$ and $n \in \mathbb{N}$ it holds that
\begin{equation}
    (x+y)^n = \sum_{k=0}^n\binom{n}{k}x^{n-k}y^k
\end{equation}
\end{theorem}

This can especially be used to expand $x^n$ around some point $x_0$
\begin{equation}
    x^n = \sum_{k=0}^n\binom{n}{k}x_0^{n-k}(x-x_0)^k,
\end{equation}
which is the core of the following proofs.

Further, we will need to estimate higher order moments of sums of bounded independent random variables.
The following lemma formalizes that these will be close to the moments of a normal distribution.

\begin{lemma}\label{lemma:central_moments}
Let $X_1, \ldots, X_n\sim X$ be i.i.d. random variables with $\E[X] = 0$, $\Var(X) = \sigma^2$. Further assume that all moments of $X$ exist.
Then the $k$-th central moment of the sample mean $\bar{X} = \frac1n \sum_{i=1}^n X_i$ can be approximated by the $k$-th moment of a normal 
distribution $N$ with mean $0$ and variance $\frac{\sigma^2}{n} = \Var(\bar X)$, namely
\begin{align*}
    \E \bar X^k = \E N^k + \mathcal{O}(n^{-\frac{k+1}{2}}).
\end{align*}
\end{lemma}

\begin{proof}
We let $Y_n = \frac{1}{\sigma \sqrt{n}} \sum_{i=1}^n X_i$, so that $\E Y_n = 0$ and $\Var(Y_n) = 1$.
A short computation reveals $\bar X = \frac{\sigma}{\sqrt{n}} Y_n$.
Further the target normal distribution $N$ can be written as $\frac{\sigma}{\sqrt{n}} Z$ with $Z\sim\mathcal{N}(0,1)$.
Now our goal is to bound
\begin{align*}
    |\E \bar X^k - \E N^k| &= \left(\frac{\sigma}{\sqrt{n}}\right)^k |\E Y_n^k - \E Z^k|
\end{align*}
Estimates of this kind are handled for example in \cite{petrov2012sums}.
We will just sketch the idea here, for a detailed treatment we refer to the book by Petrov.
The key idea is to perform a Fourier transform and bound the $k$-th derivative of the characteristic function at $0$.
For the Gaussian, its characteristic function is given by
\begin{align*}
    \varphi(t) = \E \ee^{\mathrm{i}tZ} = e^{-\frac{t^2}{2}}.
\end{align*}
and for the sum, it can be split into 
\begin{align*}
    \varphi_n(t) &= \left(\E \ee^{\mathrm{i}tY_n}\right)^n\\ 
    & = \left(\E \ee^{\mathrm{i}t\frac{X_1}{\sigma \sqrt{n}}}\right)^n \\
\end{align*}
By doing a Taylor series expansion of the exponential around $0$, we can see that the leading order term is the same
as for the Gaussian. By cancelling out the first order terms in the derivatives, we arrive at
\begin{align*}
    |\E Y_n^k - \E Z^k| &\leq \int |t|^{-1} |\varphi_n^{(k)}(t) - \varphi^{(k)}(t)| \mathrm{d}t \\
                      &\leq C_k n^{-\frac12}
\end{align*}
for some constant $C_k$. Combining this with the normalization above yields the claim.
\end{proof}

\subsubsection{Expansion of $\E[F_{i,h}^n(y)]$}
The core of the proof is the expansion of $\E[F_{i,h}^n(y)]$. We begin with a binomial expansion around $\E[F_{i,h}(y)]$:
\begin{equation}
    \E[F_{i,h}^n(y)] = \sum_{k=0}^n\binom{n}{k}\left(\E[F_{i,h}(y)]\right)^{n-k} \E\left[(F_{i,h}(y) - \E[F_{i,h}(y)])^k\right]
\end{equation}
Using Lemma~\ref{lemma:central_moments}, the central moments can be approximated (up to terms on the order of $\frac{1}{\sqrt{n}}$) by the moments of a normal distribution, simplifying the sum to:
\begin{equation}
    \E[F_{i,h}^n(y)] \approx \sum_{k=0}^{\lfloor n/2 \rfloor}\binom{n}{2k}\left(\E[F_{i,h}(y)]\right)^{n-2k} \sigma_{i,h}^{2k}(2k-1)!!
\end{equation}
where $\sigma_{i,h}^2 = \Var(F_{i,h}(y))$.

\subsubsection{Preliminary Expansions}
We require the asymptotic expansions for $\E[F_{i,h}(y)]$ and $\sigma_{i,h}^2$. Using integration by parts and Taylor series on the kernel terms, we find:
\begin{align}
    \E[F_{i,h}(y)] &= F_X(y) - \bar{h_{i}} f_X(y)\mu_{K,1} + \frac{\overline{h_{i}^2}}{2} f'_X(y)\mu_{K,2} + \Ocal(\overline{h^3}) \\
    \sigma^2_{i,h} &= \frac{1}{n_i}\left(F_X(y)(1-F_X(y)) - \bar h_i \alpha(y) + \overline{h_i^2} \beta(y) + \Ocal(\overline{h^3})\right)
\end{align}
where $\mu_{K,p}$ are kernel moments and the functions $\alpha(y)$ and $\beta(y)$ are defined as:
\begin{align*}
    \alpha(y) &= f_X(y) \mu_{K^2, 1} - 2F_X(y)f_X(y)\mu_{K,1} \\
    \beta(y) &= \frac{f'_X(y)\mu_{K^2,2}}{2} - f'_X(y)F_X(y)\mu_{K,2} - f_X^2(y)\mu_{K,1}^2
\end{align*}

\subsubsection{Helper Functions $FY_{s,N}$}
To simplify the notation, we introduce a set of helper functions $FY_{s,N}(y)$. In summation form, they are:
\begin{align*}
    FY_{0,N}(y) &= \sum_{k=0}^{\lfloor N/2 \rfloor}\binom{N}{2k}(2k-1)!! \frac{1}{n_i^k} F_X^{N-k}(y)(1-F_X(y))^k \\
    FY_{1,N}(y) &= - \sum_{k=1}^{\lfloor N/2 \rfloor}k\binom{N}{2k}(2k-1)!! \frac{1}{n_i^k} F_X^{N-k-1}(y)(1-F_X(y))^{k-1}\alpha(y) \\
    FY_{2,\alpha,N}(y) &= \sum_{k=2}^{\lfloor N/2 \rfloor}\binom{N}{2k}(2k-1)!! \frac{k(k-1)}{2n_i^k}F_X^{N-k-2}(y)(1-F_X(y))^{k-2}\alpha(y)^2 \\
    FY_{2,\beta,N}(y) &= \sum_{k=1}^{\lfloor N/2 \rfloor}\binom{N}{2k}(2k-1)!! \frac{k}{n_i^k}F_X^{N-k-1}(y)(1-F_X(y))^{k-1}\beta(y)
\end{align*}
Let $Z\sim\mathcal{N}(0,1)$ and $c = \frac{1}{\sqrt{n_i}} \sqrt{\frac{1-F_X(y)}{F_X(y)}}$.
Then the moments of $Z$ are given by $\E[Z^{2k}] = (2k-1)!!$ and $\E[Z^{2k+1}] = 0$. Using this, we can
again use the binomial theorem to simplify the sums.
\begin{align*}
    FY_{0,N}(y) &= F_X^N(y) \E\left[(1+cZ)^N\right] \\
    FY_{1,N}(y) &= -\alpha(y)\frac{N}{2\sqrt{n_i}}\frac{F_X^{N-1}(y)}{\sqrt{F_X(y)(1-F_X(y))}} \E\left[Z(1+cZ)^{N-1}\right] \\
    FY_{2,\alpha,N}(y) &= \alpha^2(y)\frac{N}{8 n_i}\frac{F_X^{N-2}(y)}{F_X(y)(1-F_X(y))} \E\left[Z \left((N-2)Z - \frac{1}{c}\right)(1+cZ)^{N-2}\right] \\
    FY_{2,\beta,N}(y) &= \beta(y)\frac{N}{2\sqrt{n_i}}\frac{F_X^{N-1}(y)}{\sqrt{F_X(y)(1-F_X(y))}} \E\left[Z(1+cZ)^{N-1}\right]
\end{align*}
We now proceed to compute an approximation of the expectations for large $n_i$, that is for small $c$.
\begin{align*}
\E\left[(1+cZ)^N\right] &= \frac{1}{\sqrt{2\pi}}\int_\R (1+cz)^N \phi(z) \dd z \\ &= \frac{1}{\sqrt{2\pi}}\int_\R \ee^{N \log{1+cz}} \phi(z) \dd z \\
&= \frac{1}{\sqrt{2\pi}}\int_\R \ee^{N \left(cz - \frac{c^2z^2}{2} + \Ocal(c^3)\right)} \phi(z) \dd z \\ 
&\approx \frac{1}{\sqrt{2\pi}}\int_\R \ee^{N \left(cz - \frac{c^2z^2}{2}\right) - \frac{z^2}{2}} \dd z
\end{align*}
as $Nc^3 \to 0$ for large $n_i$. The integral can be computed by completing the square.
\begin{align*}
    \frac{1}{\sqrt{2\pi}}\int_\R e^{-\frac{1}{2}\left((Nc^2 + 1)z^2 - 2Ncz\right)} \dd z &= \frac{e^{\frac{N^2c^2}{2(Nc^2 + 1)}}}{\sqrt{2\pi}}\int_\R e^{-\frac{1}{2}\left(\sqrt{Nc^2 + 1}z - \frac{Nc}{\sqrt{Nc^2 + 1}}\right)^2} \dd z \\
    &=  \frac{e^{\frac{N^2c^2}{2(Nc^2 + 1)}}}{\sqrt{Nc^2 + 1}} =: \rho(c, N).
\end{align*}
Note that $c = \frac{1}{\sqrt{n_i}} \sqrt{\frac{1-F_X(y)}{F_X(y)}}$, such that $Nc^2 + 1 = \frac{N(1-F_X(y)) + n_i F_X(y)}{n_i F_X(y)}$ and $1-\frac{1}{Nc^2 + 1} = \frac{N(1-F_X(y))}{N(1-F_X(y)) + n_i F_X(y)}$.

By taking the derivative, we see $\frac{\partial \rho}{\partial c} = \frac{\partial}{\partial c} \E(1+cZ)^N = N \E Z (1+cZ)^{N-1}$, the quantity we need for $Y_1$ and $Y_{2, \beta}$.
\begin{align*}
    \E Z (1+cZ)^{N-1} = \frac1N \frac{\partial \rho}{\partial c} &= \frac1N \frac{\partial}{\partial c} \frac{\ee^{\frac{N}{2}\left(1 - \frac{1}{Nc^2 + 1}\right)}}{\sqrt{Nc^2 + 1}}\\
    & = \underbrace{\frac{c(N-Nc^2-1)}{(Nc^2 +1)^2}}_{=:z_1(N,c)} \rho(N,c).
\end{align*}
In a similar fashion, we get $\E Z^2 (1+cZ)^{N-2}$ by taking the derivative twice:
\begin{align*}
    \E Z^2 (1+cZ)^{N-2} = \frac1{N(N-1)} \frac{\partial^2 \rho}{\partial c^2} = \underbrace{\frac{2N^3c^6 + (3N^2 - 5N^3)c^4 + (N^3 - 4N^2)c^2 + N - 1}{(N-1)(Nc^2 + 1)^4}}_{=:z_2(N,c)}\rho(N,c).
\end{align*}

\subsubsection{Expansion Coefficients for $\E[F_{i,h}^N(y)]$}

Let $\E[F_{i,h}^N(y)] = E_{0,N}(y) + h_i E_{1,N}(y) + h_i^2(E_{2a,N}(y) + E_{2b,N}(y)) + \Ocal(h^3)$. The coefficients are:

\begin{align*}
    E_{0,N}(y) &= FY_{0,N}(y) \\
    E_{1,N}(y) &= FY_{1,N}(y) - N f_X(y)\mu_{K,1}FY_{0,N-1}(y) \\
    E_{2a,N}(y) &= FY_{2,\beta,N}(y) + \frac{N}{2}f'_X(y)\mu_{K,2}FY_{0,N-1}(y) \\
    E_{2b,N}(y) &= FY_{2,\alpha,N}(y) - N f_X(y)\mu_{K,1}FY_{1,N-1}(y) + \binom{N}{2}f_X^2(y)\mu_{K,1}^2FY_{0,N-2}(y)
\end{align*}

A reasonable assumption would be $\mu_{K,1} = 0$, which simplifies the expressions for $E_{1,N}(y)$ and $E_{2b,N}(y)$ as well as $\alpha$ and $\beta$ considerably, a fact we will use in the following section.

\subsection{Final Bias and Variance Coefficients}

\subsubsection{The Bias Coefficients}
According to the previous section, the bias coefficients are given by
\begin{align*}
    b_{0,i}(y) &= E_{0,n_i}(y) - F_X^{n_i}(y) & b_{1,i}(y) &= E_{1,n_i}(y) \\
    b_{2a,i}(y) &= E_{2a,n_i}(y) & b_{2b,i}(y) &= E_{2b,n_i}(y)
\end{align*}

From now on we assume the large $n$ limit and a kernel $K$ centered at $0$, i.e. $\mu_{K,1} = 0$.
We then have the following approximations for the bias coefficients:
\begin{align*}
    b_{0,i}(y) & = FY_{0,n_i}(y) - F_X^{n_i}(y) \approx  F_X^{n_i}(y) (\rho(n_i, c) - 1)
\end{align*}
When inserting $n_i$ into $\rho$, the expression simplifies considerably, namely $\rho(n_i, c) = \sqrt{F_X} \ee^{\frac{n_i}2 (1-F_X)}$. Thus, we have
\begin{align*}
    b_{0,i}(y) & \approx  F_X^{n_i}(y) (\sqrt{F_X} \ee^{\frac{n_i}2 (1-F_X)} - 1)
\end{align*}

\begin{align*}
    b_{1,i}(y) & = FY_{1,n_i}(y) \approx  -\alpha(y)\frac{n_i}{2\sqrt{n_i}}\frac{F_X^{n_i-1}(y)}{\sqrt{F_X(y)(1-F_X(y))}} \left(\frac{n_ic}{n_ic^2+1} - 1\right)\frac{\rho(n_i,c)}{n_ic^2+1} \\
    &= -\alpha(y)\frac{\sqrt{n_i}}{2}\frac{F_X^{n_i}(y)}{\sqrt{(1-F_X(y))}} \left(\sqrt{n_i F_X (1-F_X)} - 1\right) \ee^{\frac{n_i}2 (1-F_X)}
\end{align*}

\begin{align*}
    b_{2a,i}(y) & = FY_{2,\beta,n_i}(y) + \frac{n_i}{2}f'_X(y)\mu_{K,2}FY_{0,n_i-1}(y) \\
    &\approx \beta(y)\frac{\sqrt{n_i}}{2}\frac{F_X^{n_i}(y)}{\sqrt{(1-F_X(y))}}\left(\sqrt{n_i F_X (1-F_X)} - 1\right) \ee^{\frac{n_i}2 (1-F_X)}+ \\&
    + \frac{n_i}{2}f'_X(y)\mu_{K,2} F_X^{n_i-1}(y) \sqrt{F_X} \ee^{\frac{n_i-1}2 (1-F_X)}
\end{align*}

\begin{align*}
    b_{2b,i}(y) & = FY_{2,\alpha, n_i} \\&\approx \frac{\alpha^2(y)}{8}\frac{F_X^{n_i-3}(y)}{1-F_X(y)} \left((n_i-2)z_2(n_i, c)\rho(n_i, c) - \frac{1}{c}z_1(n_i-1, c)\rho(n_i-1, c)\right)
\end{align*}

\subsubsection{The Variance Coefficients}
The variance coefficients are given by

\begin{align*}
    V_{0,i}(y) &= E_{0,2n_i}(y) - E_{0,n_i}(y)^2 \\
    V_{1,i}(y) &= E_{1,2n_i}(y) - 2E_{0,n_i}(y)E_{1,n_i}(y) \\
    V_{2a,i}(y) &= E_{2a,2n_i}(y) - 2E_{0,n_i}(y)E_{2a,n_i}(y) \\
    V_{2b,i}(y) &= E_{2b,2n_i}(y) - 2E_{0,n_i}(y)E_{2b,n_i}(y) - E_{1,n_i}(y)^2
\end{align*}

By substituting the large-$n_i$ approximations for the coefficients $E_{s,N}(y)$ (assuming a centered kernel with $\mu_{K,1}=0$), we can write these variance coefficients in a form analogous to the bias coefficients.

The zeroth-order coefficient is:
\begin{align*}
    V_{0,i}(y) & \approx F_X^{2n_i}(y) \left( \rho(c, 2n_i) - \rho(c, n_i)^2 \right)
\end{align*}

The first-order coefficient depends on $\alpha(y) = f_X(y) \mu_{K^2, 1}$:
\begin{align*}
    V_{1,i}(y) & \approx -\frac{\alpha(y) n_i F_X^{2n_i-1}(y)}{\sqrt{n_i F_X(y)(1-F_X(y))}} \left( z_1(2n_i, c)\rho(c, 2n_i) - z_1(n_i, c)\rho(c, n_i)^2 \right)
\end{align*}

The second-order coefficients are split into two parts. The first part, $V_{2a,i}$, depends on $\beta(y) = \frac{f'_X(y)\mu_{K^2,2}}{2} - f'_X(y)F_X(y)\mu_{K,2}$:
\begin{align*}
    V_{2a,i}(y) \approx& \frac{\beta(y) n_i F_X^{2n_i-1}(y)}{\sqrt{n_i F_X(y)(1-F_X(y))}} \left( z_1(2n_i, c)\rho(c, 2n_i) - z_1(n_i, c)\rho(c, n_i)^2 \right) \\
    &+ n_i f'_X(y)\mu_{K,2} F_X^{2n_i-1}(y) \left( \rho(c, 2n_i-1) - \rho(c, n_i)\rho(c, n_i-1) \right)
\end{align*}

The second part, $V_{2b,i}$, depends on $\alpha(y)^2$. Let 
\[\mathcal{Z}(N,c) = \left((N-2)z_2(N, c)\rho(N, c) - \frac{1}{c}z_1(N-1, c)\rho(N-1, c)\right).\] Then
\begin{align*}
    V_{2b,i}(y) \approx& \frac{\alpha^2(y) F_X^{2n_i-2}(y)}{4 F_X(y)(1-F_X(y))} \left( \mathcal{Z}(2n_i,c) - \rho(c,n_i)\mathcal{Z}(n_i,c) \right) \\
    &- \frac{\alpha^2(y) n_i F_X^{2n_i-2}(y)}{F_X(y)(1-F_X(y))} z_1(n_i,c)^2\rho(n_i,c)^2
\end{align*}
These expressions, while complex, provide the complete second-order expansion of the variance needed to compute the MISE and the optimal bandwidth.

\section{Proof of Theorem~\ref{thm:optimizer}}\label{appendix:optimization}

In order to derive the expressions for the matrix $\mathbf{Q}$ and the vector $\mathbf{c}$, we need to compute the 
derivative of the MSE with respect to $h$. We will compute the derivative with respect to $h_{ij}$.

Note that 
\[\frac{\partial \bar h_k}{\partial h_{ij}} = \frac1{n_i}\delta_{ik},\]
\[\frac{\partial \bar{h_k^2}}{\partial h_{ij}} = \frac{2h_{ij}}{n_i}\delta_{ik},\]
and
\[\frac{\partial \bar{h_k}^2}{\partial h_{ij}} = \frac{2 \bar h_i}{n_i}\delta_{ik}\]

We can now use these expressions in the derivative of the MSE up to order $h^3$. Further let $b_0:=\sum_i b_{0,i}$.
\begin{align*}
\frac{\partial\MSE}{\partial h_{ij}} &= \frac{1}{m^2} \left[ 2 b_0 \left(\sum_k \frac{\partial \bar h_k}{\partial h_{ij}} b_{1,k}\right) + \sum_k \frac{\partial \bar h_k}{\partial h_{ij}}  V_{1,k} \right] \\ 
            &\quad+ \frac{1}{m^2} \Bigg[ 2\left(\sum_k \bar{h_k} b_{1,k}\right)\left(\sum_k\frac{\partial \bar h_k}{\partial h_{ij}} b_{1,k}\right) \\
            &\qquad\qquad + 2 b_0 \sum_k \left(\frac{\partial \bar{h_k^2}}{\partial h_{ij}} b_{2a,k} + \frac{\partial \bar{h_k}^2}{\partial h_{ij}^2} b_{2b,k}\right) \\
            &\qquad\qquad + \sum_k \left( \frac{\partial \bar{h_k^2}}{\partial h_{ij}} V_{2a,k} + \frac{\partial \bar{h_k}^2}{\partial h_{ij}^2} V_{2b,k} \right) \Bigg]\\
            &=\frac{1}{m^2n_i} ( 2 b_0 b_{1,i} + V_{1,i} ) \\ 
            &\quad+ \frac{1}{m^2n_i} \Bigg[ 2\left(\sum_k \bar{h_k} b_{1,k}\right)b_{1,i} \\
            &\qquad\qquad + 4 b_0 h_{ij} b_{2a,i} + \bar{h_i} b_{2b,i}\\
            &\qquad\qquad + 2 h_{ij} V_{2a,i} + \bar{h_i} V_{2b,i}\Bigg]
\end{align*}
By setting this to zero, we obtain the system of equations
\begin{align*}
 2\left(\sum_k \bar{h_k} b_{1,k}\right)b_{1,i} + 4 b_0 (h_{ij} b_{2a,i} + \bar{h_i} b_{2b,i}) + 2 h_{ij} V_{2a,i} + \bar{h_i} V_{2b,i} = - (2 b_0 b_{1,i} + V_{1,i})
\end{align*}
Now notice that all coefficients are independent of $j$, so the solutions will also be independent of $j$, which means they are blockwise constant. This shows the final statement of the theorem. Taking this into account (by setting $\bar h_i = h_{ij} = h_i$), we see that
\begin{align*}
 \underbrace{2\left(\sum_k h_k b_{1,k}\right)b_{1,i} + 4 b_0 h_i (b_{2a,i} + b_{2b,i}) + 2 h_i (V_{2a,i} + V_{2b,i})}_{=2\mathbf{Q}_i \mathbf{h}} = - \underbrace{(2 b_0 b_{1,i} + V_{1,i})}_{=\mathbf{c}_i}
\end{align*}
By introducing the notation $b_{2,i} = b_{2a,i} + b_{2b, i}$ and $V_{2,i} = V_{2a,i} + V_{2b,i}$, we arrive at the statement of the theorem.

\section{Proof of Theorem~\ref{theorem:stability}}\label{appendix:stability}

Lemma~\ref{lemma:positive_definite} and the expressions derived in the proof of Theorem~\ref{thm:asymptotic_expansion} already provide the needed tools to establish the stability of the system.
In this section we always denote the number of blocks with $m$ and the block size with $n$, as all blocks are assumed to have the same size.
Our goal is to derive an asymptotic (in $n$) closed form expression for $D$, as defined in Lemma~\ref{lemma:positive_definite}.
For an asymptotic analysis, we assume large $n$, such that we only need to consider the dominant terms (with respect to $n$) in each expression.
In this case, we can also assume $n\approx n-1 \approx n-2$ to simplify the expressions.

Recall 
\[
  D = \int_{-\infty}^\infty 2m b_0(y) b_2(y) + V_2(y) \dd y
\]

We will consider the optimization of the $q$-MISE, that is we perform the integration starting from the $q$-quantile of the base distribution $F_X$.
Further, we will assume that $F_X$ is heavy tailed with parameter $\gamma$,
this means we consider 
\[
  D_q = \int_{F_X(y)\geq q} 2m b_0(y) b_2(y) + V_2(y) \dd y
\]

The determination of the sign of $D_q$ involves very intricate calculations and the number of terms in the individual integrals is quite large. 
We will resort to the asymptotic behaviour (in terms of $n$) and follow the following general strategy:
\begin{enumerate}
\item {\bf Scaling check:} We will first check each additive component and see how it scales with $n$.
     We can do this without fully performing the integration, but we need to take the behaviour of the integrand into accound.
     As a first step, we will establish some intuition about how this scaling can be estimated via \emph{Watson's lemma}.
\item {\bf Dominant term:} We will then identify the dominant term in the integral (i.e. the term that grows fastest with $n$).
    For large $n$, this will dominate and give the asymptotic behaviour of the integral.
\end{enumerate}

The next subsection will introduce the methods we will use to perform the scaling check and the identification of the dominant term.

\subsection{Used methods}

\subsubsection{Watson's lemma}

Watson's lemma is a powerful tool for asymptotic analysis of integrals, especially when dealing with integrals of the form
\[
    I(n) = \int_0^T f(t) e^{-n t} \dd t
\]
where $f(t)$ is a function that behaves like a polynomial for large $t$ and $T$ can be a positive constant or infinity.
We want to look at the limiting behaviour of $I(n)$ as $n \to \infty$.
Intuitively, for large $n$ the exponential term $e^{-n t}$ makes sure that the integral is dominated by the behaviour of $f(t)$ near $t=0$.

\begin{theorem}[Watson's lemma, \cite{miller2006applied}]\label{thm:watson}
    Let $0<T\leq \infty$. $f(t)$ be a function such that $f(t) = t^k g(t)$ for some $k > -1$ where $g(t)\in\mathcal{C}^\infty(B)$ for some open set $B$ containing $[0, T]$ with $g(0)\neq 0$.
    Additionally we assume $|f(t)| < K \ee^{bt}$ for all $t>0$ and some constants $K, b > 0$ or $\int_0^T |f(t)| \dd t < \infty$.
    Then
    \[\left|\int_0^T \ee^{-nt} f(t) \dd t\right| < \infty\]
    and
    \[
    \int_0^T \ee^{-nt} f(t) \dd t \asymp_n \sum_{j=0}^\infty \frac{g^{(j)}(0) \Gamma(k+j+1)}{j! n^{k+j+1}}.
    \]
\end{theorem}

\begin{proof}
See \cite{miller2006applied}, Chapter 2.2 for a proof.
\end{proof}

For large $n$ and if $g$ is independent of $n$, we see that the sum on the right-hand side is dominated by the first term, which gives us
\[\int_0^T \ee^{-nt} f(t) \dd t \asymp_n \frac{g(0) \Gamma(k+1)}{n^{k+1}}.\]

This means that once we have identified the leading term in the integrand, we can use Watson's lemma to compute the asymptotic behaviour of the integral.

{\bf Heuristic:} If the leading order of the integrand is $k$, then the integral behaves like $\frac{1}{n^{k+1}}$ (up to constants).

Further, we will need a modification of the above integral to account for a different power of $t$ in the exponential.
\begin{lemma}\label{lemma:mod_watson}
Consider the same setting as in Theorem~\ref{thm:watson}. Then for $c>0$, asymptotically, we have
\[\int_{0}^{T} \ee^{-cnt^2} f(t) \dd t \asymp_n \sum_{j=0}^\infty \frac{g^{(j)}(0)}{j!} \frac{\Gamma\left(\frac{j+\gamma+1}{2}\right)}{2 (cn)^{\frac{j+\gamma+1}{2}}}.\]
\end{lemma}
\begin{proof}
We will perform an asymptotic expansion of the integral in the spirit of Laplace's method in the proof of Watson's lemma. 
We will use a Taylor expansion of $g$ around $0$ to rewrite $f$ as 
\[f(t) = \sum_{j=0}^\infty \frac{g^{(j)}(0)}{j!} t^{j+\gamma}.\]
The change of variables $u = cnt^2$ then gives us

\begin{align*}
\int_{0}^{T} \ee^{-cnt^2} f(t) \dd t &= \int_{0}^{T} \ee^{-cnt^2} \sum_{j=0}^\infty \frac{g^{(j)}(0)}{j!} t^{j+\gamma}\dd t\\
&= \sum_{j=0}^\infty \frac{g^{(j)}(0)}{j!} \int_{0}^{T} \ee^{-cnt^2} t^{j+\gamma}\dd t\\
&\approx \sum_{j=0}^\infty \frac{g^{(j)}(0)}{j!} \int_{0}^{\infty} \ee^{-cnt^2} t^{j+\gamma}\dd t\\
&= \sum_{j=0}^\infty \frac{g^{(j)}(0)}{j!} \int_{0}^{\infty} \ee^{-u} \left(\frac{u}{cn}\right)^{\frac{j+\gamma}{2}} \frac{1}{2\sqrt{cn u}} \dd u\\
&= \sum_{j=0}^\infty \frac{g^{(j)}(0)}{j!} \frac{1}{2 (cn)^{\frac{j+\gamma+1}{2}}} \int_{0}^{\infty} \ee^{-u} u^{\frac{j+\gamma-1}{2}} \dd u\\
&= \sum_{j=0}^\infty \frac{g^{(j)}(0)}{j!} \frac{\Gamma\left(\frac{j+\gamma+1}{2}\right)}{2 (cn)^{\frac{j+\gamma+1}{2}}}.
\end{align*}
Where the approximation in the fourth line is valid for large $n$ as the exponential decay makes the contribution of the tail negligible.
\end{proof}

\subsubsection{Change of variables}

Now let us introduce a change of variables that is at the heart of the following derivation. 
We set $F_X(y) = 1 - \frac{z}{n}$.
While this might not be completely intuitive, let us examine the effect on the recurring terms
$F_X^n$ and $\ee^{\frac{n}{2} (1-F_X)}$. We have
\[
  F_X^n(y) = \left(1 - \frac{z}{n}\right)^n \approx \ee^{-z}
\]
for large $n$ as well as 
\[
  \ee^{\frac{n}{2}(1-F_X(y)} = \ee^{\frac{z}{2}}.
\]
Within the integral, this changes the integration boundaries from $[q, \infty]$ to $[0, n(1-q)]$ and the Jacobian is $\frac{\dd z}{\dd y} = - n f_X(y)$.
Note that every term in the integral is proportional to $f_X^2$, so we can factor out $f_X$ and use it for the change of variables.

Similarly, the change of variables $F_X(y)= 1-t$ (or $z=nt$) can be done and leads to the same simplifications, but now the integration boundaries are $[0, 1-q]$ and the Jacobian is $\dd t = -\frac{\dd z}{n} = -\frac{\dd y}{f_X(y)}$.
The exponential terms in turn become $\ee^{\frac{n}{2} t}$ and $\ee^{-nt}$ (where we retained the approximation for large $n$).

\subsubsection{Asymptotic approximation for heavy tailed distributions}

We assume that the base distribution $F_X$ is heavy tailed with parameter $\gamma > 0$.
Many of the results in this section are based on the theory of extreme value distributions, see \cite{haan2006extreme} for a comprehensive introduction.

In the considered integrals we have terms either proportional to $f'_X$ or $f^2_X$. Given that we are in the
asymptotic regime for $F_X$ (i.e. $q$ is large), \cite{haan2006extreme}, by combining Theorem 1.1.8 and Remark 1.2.8 we get
\[\lim_{t\to x^*} \frac{(1-F_X(t))f'_X(t)}{f^2_X(t)} = -\gamma-1\]
if $F_X$ is sufficiently regular (i.e. $f_X$ and $f'_X$ exist), which shows that in this case $f'_X$ and $f^2_X$ are closely related.
Further, \cite{haan2006extreme}, Theorem 1.2.1 gives necessary and sufficient conditions for $F_X$ to be heavy tailed with parameter $\gamma$.

This means that we asymptotically have
\begin{align}\label{eq:asymptotic_f_prime}
    f'_X(y) \asymp -(\gamma+1) \frac{f^2_X(y)}{(1-F_X(y))}.
\end{align}

We will first work through the asymptotic approximation for Frechet class distributions.
Theorem 1.2.1 in~\cite{haan2006extreme} states that in order for the distribution $F_X$ to be in the 
Frechet class, $1-F$ must be regularly varying at infinity with index $-\frac1\gamma$, where $\gamma > 0$. The representation
theorem for regularly varying functions, Theorem B.1.6 in~\cite{haan2006extreme}, then gives us the
approximation
\begin{align}\label{eq:reg_var_approx}
    1-F_X(y) \approx C y^{-\frac1\gamma}
\end{align}
for large $y$ with a positive constant $C$.
Using this, the density is then given by the derivative
\[f_X(y) \approx -\frac{d}{dy} \left( C y^{-\frac1\gamma}\right) = \frac{C}{\gamma} y^{-(\frac1\gamma+1)}.\]
This approximation is again valid for large $y$. By multiplying Equation~\eqref{eq:reg_var_approx} by $n$, we see
\[z= n C y^{-\frac1\gamma}\]
or equivalently
\[y = \left(\frac{nC}{z}\right)^{\gamma}\]
so we arrive at 
\begin{align}
    f_X(z) \approx \frac{C}{\gamma} \left(\frac{nC}{z}\right)^{-\gamma(\frac1{\gamma}+1)} = \frac{\tilde C}{\gamma}\left(\frac{z}{n}\right)^{\gamma+1}
\end{align}
with a new constant $\tilde C$. 

For a Weibull class distribution there is a similar result given in Theorem 1.2.1 in~\cite{haan2006extreme}, which states that the distribution is in the domain of attraction for an extreme value distribution $G_\gamma$ for $\gamma < 0$ if
\begin{enumerate}
\item Its endpoint $x^*$ is finite, i.e. $\lim_{x\to x^*} F_X(x) = 1$.
\item The distribution is regularly varying at $x^*$ with index $-\frac1\gamma$, which means \[\lim_{t\downarrow 0} \frac{1-F(x^* - tx)}{1-F(x^* - t)} = x^{-\frac1\gamma}.\]
\end{enumerate}
This is similar to the Frechet case, but now we have a finite endpoint $x^*$.
In the same spirit, we can use the representation theorem (this time not around infinity but around zero) to obtain
\[1-F_X(x^* - t) \approx C t^{-\frac1\gamma}\]
for $t$ close to zero, which in turn leads to 
\[1-F_X(y) \approx C (x^* - y)^{-\frac1\gamma}\]
for $y$ close to $x^*$. 
The density is then given by the derivative
\[f_X(y) \approx -\frac{d}{dy} \left( C (x^* - y)^{-\frac1\gamma}\right) = -\frac{C}{\gamma} (x^* - y)^{-\frac1\gamma - 1}.\]
Note that for $\gamma < 0$ this expression is positive, as required. The sign is absorbed into the constant $\tilde{C}$ in the expression below.
By again using the substitution $z = n(1-F_X(y))$, we can rewrite this as
\[\frac{z}{n} \approx C (x^* - y)^{-\frac1\gamma}\] or equivalently
\[x^* - y \approx \left(\frac{z}{nC}\right)^{-\gamma}.\]
Plugging this into the density approximation, we arrive at
\begin{align}
    f_X(z) \approx \frac{C}{\gamma} \left(\frac{z}{nC}\right)^{\gamma(\frac1{\gamma}+1)} = \frac{\tilde C}{\gamma}\left(\frac{z}{n}\right)^{\gamma+1}
\end{align}
with a new constant $\tilde C$. This is the same expression as in the Frechet case, but now we have a finite endpoint $x^*$ and negative $\gamma$.

\subsubsection{The Gumbel Case ($\gamma=0$)}

The Gumbel class represents the boundary case where $\gamma=0$. As established in our prior discussion, these distributions do not have a simple power-law tail but rather one that decays exponentially. Consequently, there is no single universal monomial for $f_X(z)$. However, we can derive the form for the most important members.

A sufficient condition for a distribution to be in the Gumbel domain of attraction is again given in Theorem 1.2.1 of~\cite{haan2006extreme}. 
In this case, $x^*$ can be both finite or infinite and \[\lim_{x\to x^*} \frac{1-F_X(x+t a(x))}{1-F_X(x)} = e^{-t}\] for all real $t$. 
The function $a$ is a suitable positive function.
If it exists, the following is finite and $a$ can be chosen to be 
\[a(x) = \frac{\int_x^{x^*} 1-F(s) \dd s}{1-F(x)}\]
for $x < x^*$. 
Theorem 1.2.6 in~\cite{haan2006extreme} gives another representation for $1-F_X(x)$, namely there exist positive functions $c$ and $d$, with $d$ continuous,
such that for all $x\in(x_0, x^*)$ we have 
\[1-F_X(x) = c(x)\exp\left(-\int_{x_0}^{x} \frac{\dd s}{d(s)}\right)\]
where $\lim x\to x^* c(x) = c\in(0,\infty)$ and $\lim_{x\to x^*} d'(x) = 0$ as well as $\lim_{x\to x^*} d(x) = 0$.

Taking the derivative of this expression, we find the density function
\[f_X(x) = \left(\frac{1}{d(x) - \frac{c'(x)}{c(x)}}\right) (1-F_X(x)).\]

The condition on $c$ implies that it is slowly varying and Theorem B.1.6 in~\cite{haan2006extreme} then implies that $\frac{c'(t)}{c(t)} = o(\frac{1}{t})$,
so it goes to zero faster than $\frac{1}{t}$ as $t\to\infty$.
Typically the term $\frac{1}{d(s)}$ is will decay slower than this and thus dominates the expression, so asymptotically we have
\[f_X(x) \approx \frac{1-F_X}{d(x)}\]
for large $x$. The exact form of $d(x)$ depends on the specific distribution, but it is often a slowly varying function, we will
look at two examples below.

\paragraph{1. The Exponential Distribution:}
For an exponential distribution with rate $\lambda$, we have $1-F_X(y) = e^{-\lambda y}$ and $f_X(y) = \lambda e^{-\lambda y}$.
Our substitution $z = n(1-F_X(y))$ becomes $z = n e^{-\lambda y}$.
By directly substituting this into the expression for the density, we get:
\[f_X(z) = \lambda \left(\frac{z}{n}\right)\]

\paragraph{2. The Normal Distribution:}
The standard normal distribution is also in the Gumbel domain of attraction. 
The tail can be approximated for large $y$ as $1-F_X(y) \approx \frac{1}{y\sqrt{2\pi}}e^{-y^2/2}$. 
The density is $f_X(y) = \frac{1}{\sqrt{2\pi}}e^{-y^2/2}$.
The substitution $z = n(1-F_X(y))$ leads to the leading-order relationship $y \approx \sqrt{2\ln(n/z)}$.
Substituting this back into the density gives:
\[f_X(z) \approx \frac{1}{\sqrt{2\pi}}\exp\left(-\frac{2\ln(n/z)}{2}\right) = \frac{1}{\sqrt{2\pi}}\exp(-\ln(n/z)) = \frac{1}{\sqrt{2\pi}}\left(\frac{z}{n}\right)\]

For both of these canonical examples in the Gumbel class, the density function in terms of $z$ takes the linear form $f_X(z) \propto z/n$.
This is the same functional form we would get by formally setting $\gamma=0$ in the general expression $\frac{\tilde C}{\gamma}\left(\frac{z}{n}\right)^{\gamma+1}$,
confirming the consistency of our approximations across the boundary of the classes.
However, it is important to note that in general, this representation will depend on the exact form of the function $d$. 
We will for the sake of simplicity assume this functional form for the Gumbel class, but bear in mind that it is not a universal representation. 

\subsubsection{Integration procedure}

The computation of the complete integral is involved, but follows along a clear line of reasoning for each component.
We will first split the integral into its components, i.e. we will consider the integrals of $b_0$, $b_2$, and $V_2$ separately.

Then we will perform the change of variables to $z$ and use the asymptotic approximations for the density function $f_X(z)$.
Further, the derivative of $f_X$ will be approximated using Equation~\eqref{eq:asymptotic_f_prime}.
This will lead to integrals of the form found in Watson's lemma, which we can then analyze by finding the leading order term in the integrand.
Application of Watson's lemma shows that we then only need to compute $g(0)$ and this already determines the integral asymptotically.

\subsection{Asymptotic Analysis of $D_q$}

\subsubsection{The $b$ integral}

We have
\begin{align}
    b_0(y) &= F_X^{n}(y) (\sqrt{F_X} \ee^{\frac{n}2 (1-F_X)} - 1),
\end{align}
by rewriting this in terms of $t = 1-F_X$ and using the approximation $(1-t)^{n} \approx \ee^{-nt}$ for large $n$, we get
\begin{align*}
    b_0(z) &= \ee^{-nt}\left(\sqrt{1-t} \ee^{\frac{nt}{2}} - 1\right)
\end{align*}

Next we perform the change of variables for
\begin{align}
    b_{2}(y) &= \beta(y)\frac{\sqrt{n}}{2}\frac{F_X^{n}(y)}{\sqrt{(1-F_X(y))}}\left(\sqrt{n F_X (1-F_X)} - 1\right) \ee^{\frac{n}2 (1-F_X)}+ \\&
    + \frac{n}{2}f'_X(y)\mu_{K,2} F_X^{n-1}(y) \sqrt{F_X} \ee^{\frac{n-1}2 (1-F_X)} \\&
    + \frac{\alpha^2(y)}{8}\frac{F_X^{n-3}(y)}{1-F_X(y)} \left((n-2)z_2(n, c)\rho(n, c) - \frac{1}{c}z_1(n-1, c)\rho(n-1, c)\right)
\end{align}
This part is a bit more work to simplify. 
We have 
\[z_1(n, c) = \frac{c(n-nc^2-1)}{(nc^2 +1)^2},\]
where $c= \frac{1}{\sqrt{n}} \sqrt{\frac{1-F_X}{F_X}} = \frac{1}{\sqrt{n}} \sqrt{\frac{t}{1-t}}$.
Note that $nc^2 = \frac{1-F_X}{F_X} = \frac{t}{1-t}$ and $nc^2+1 = \frac1{F_X} =\frac{1}{1-t}$, which is independent of $n$.
We obtain
\[\frac1c z_1(n,t) = (1-t)(n(1-t)-1).\]

Also note that
\[\rho(n, t) = \sqrt{1-t}\ee^{\frac{nt}{2}}.\]

Continuing with the next coefficient, we see
\begin{align*}
    z_2(n) &= \frac{2n^3c^6 + (3n^2 - 5n^3)c^4 + (n^3 - 4n^2)c^2 + n - 1}{(n-1)(nc^2 + 1)^4}\\
        &= \frac{2(nc^2)^3+ (3 - 5n) (nc^2)^2+ (n^2-4n)nc^2 + n - 1}{(n-1)(nc^2 + 1)^4}\\
\end{align*}
Plugging in the relations for $nc^2$ we can rewrite the expression as 
\[
    z_2(n, t) = \frac{2(1-t)t^3 + (3-5n)(1-t)^2t^2 + n(n-4)t(1-t)^3 + (n-1)(1-t)^4}{(n-1)},
\]
which is simply a polynomial in $t$ (as is $\frac1c z_1$). We set $(n-2) z_2 - \frac1c z_1 =: p(n,t)$. 
Note especially that $p(n, 0) = -1$, a fact that we will use in the application of Watson's lemma.

Using the representation $f'_X(t) = -\frac{(\gamma + 1)}{t} f^2_X(t)$ in the asymptotic range, we see that
\begin{align*}
    \beta(t) &= \mu_{K, 2}f'_X(\underbrace{\frac{\mu_{K^2, 2}}{2\mu_{K, 2}}}_{=:\kappa} - F_X)\\
             &\approx -\mu_{K,2} \frac{(\gamma + 1)}{t} f^2_X(t) \left(\kappa - 1 +t\right).
\end{align*}

So finally 
\begin{align*}
b_2 &= -\mu_{K,2} \frac{(\gamma + 1)}{t} f^2_X(t) \ee^{-\frac{nt}{2}}\underbrace{\left(\frac{\sqrt{n}}{2\sqrt{t}}\left(\sqrt{nt(1-t)} - 1\right)(\kappa -1 +t) + \frac{n}2 \sqrt{1-t}\right)}_{=:q(n,t)} +\\
&+ \frac{\mu_{K^2,1}^2 f^2_X(t)}{8}\frac{\ee^{-\frac{nt}{2}}}{t} p(n, t)
\end{align*}

Again note that we have $q(n,t) = \frac1{\sqrt{t}} g(n,t)$ with $g(n,0) = \frac{\sqrt{n}}{2} (1-\kappa)$.

Putting all this together, we arrive at the first integral
\begin{align}\label{eq:b2_integral}
\int_{q_F}^\infty b_0(y) b_2(y) \dd y &= \int_0^{(1-q)}b_0(t)b_2(t)\frac{\dd t}{f_X(t)}\\
 &= \int_{0}^{(1-q)} \frac{f_X(t)}{t} \ee^{-\frac{3nt}{2}}\left(\sqrt{1-t} \ee^{\frac{nt}{2}} - 1\right)\left(\frac{\mu^2_{K^2, 1}}{8}p(n,t) - \mu_{K,2}(\gamma+1)q(n,t)\right)\dd t
\end{align}

We now use the approximations for $f_X(t)$ established in the previous section.
For a heavy tailed distribution with parameter $\gamma\neq 0$, we have
\[f_X(z) \approx D t^{\gamma+1},\]
where $D$ is a constant that depends on the distribution. We assume that this also holds for $\gamma=0$, but depending on the distribution this might not be the case.

By plugging this into Equation~\eqref{eq:b2_integral}, setting $q = \frac{g}{\sqrt{t}}$ and simplifying we find
\begin{align*}
    \int_{q_F}^\infty b_0(y) b_2(y) \dd y &= \int_{0}^{(1-q)} D t^{\gamma-\frac12} \ee^{-\frac{3nt}{2}}\left(\sqrt{1-t} \ee^{\frac{nt}{2}} - 1\right)\left(\sqrt{nt}\frac{\mu^2_{K^2, 1}}{8}p(n,t) - \mu_{K,2}(\gamma+1)g(n,t)\right)\dd t\\
\end{align*}

We will use symbolic computation (SymPy\footnote{The SymPy code used for this computation and the one in the following section is available at \url{https://github.com/DataLabHell/ddevd_symbolic}}
) to evaluate the leading order term of this integral, as the expression is quite involved.
We will shortly describe our approach here.
First, all polynomial terms in $t$ are expanded. By linearity of the integral, we can then consider each term separately.
Then, we apply Watson's lemma (including the derivatives) to each term. Then we sum up all contributions to get the final leading order term.
The result of this computation yields the leading order asymptotic behavior for the $b_0 b_2$ integral:
\begin{align*}
\int_{q_F}^\infty b_0(y) b_2(y) \dd y \asymp - C_b n^{-\gamma}
\end{align*}
Note that the application of Watson's lemma hinges on the fact that $\gamma > -\frac12$, because otherwise the exponent in the integral would be smaller than $-1$.

\subsection{The Variance Integral}

The computation for the second part of the integral containing the term $V_2$ works in a similar fashion.
Recall
\begin{align}
    V_{2}(y) \approx& \frac{\beta(y) n F_X^{2n-1}(y)}{\sqrt{n F_X(y)(1-F_X(y))}} \left( z_1(2n, c)\rho(c, 2n) - z_1(n, c)\rho(c, n)^2 \right)\label{eq:firstline_v} \\
    &+ n f'_X(y)\mu_{K,2} F_X^{2n-1}(y) \left( \rho(c, 2n-1) - \rho(c, n)\rho(c, n-1) \right) \label{eq:secondline_v}\\
    &+ \frac{\alpha^2(y) F_X^{2n-2}(y)}{4 F_X(y)(1-F_X(y))} \left( \mathcal{Z}(2n,c) - \rho(c,n)\mathcal{Z}(n,c) \right) \label{eq:thirdline_v}\\
    &- \frac{\alpha^2(y) n F_X^{2n-2}(y)}{F_X(y)(1-F_X(y))} z_1(n,c)^2\rho(n,c)^2.\label{eq:fourthline_v}
\end{align}

In the following integrals, terms of the form $\ee^{-\frac{ct^2}{1+t}}$ will appear. To bring them in the standard form for Watson's lemma, we
observe that $\frac{-t^2}{1+t} = \frac{1-t^2}{1+t} - \frac{1}{1+t} = 1 - t - \frac{1}{1+t}$, so we can rewrite the exponential as
$\ee^{-\frac{ct^2}{1+t}} = \ee^{-ct}\ee^{\frac{ct}{1+t}}$, where the second factor one for $t=0$.

This expression is even more complicated than the one for $b_0b_2$, however we can use the same asymptotic approximations for $f_X$ and $f'_X$ as above.
As a first step, we note all expressions occurring in the integral in terms of $t$, as in the previous section.
We have
\begin{align*}
    f_X(z) &\approx D t^{\gamma+1},\\
    \alpha(z) &= D \mu_{K^2, 1} t^{\gamma+1},\\
    \beta(z) &= -\mu_{K,2}\frac{(\gamma+1)}{t} f^2_X(t) \left(\kappa - 1 +t \right), \\
    \rho(c, n) &= \sqrt{1-t}\ee^{\frac{nt}{2}} \approx \rho(c, n-1),\\
    \rho(c, 2n) &= \sqrt{\frac{1-t}{1+t}}\ee^{\frac{2nt}{1+t}} \approx \rho(c, 2n-1),\\
    z_1(n, c) &= \sqrt{t(1-t)}\left(\sqrt{n}(1-t) - \frac1{\sqrt{n}}\right),\\
    z_1(2n, c) &= \sqrt{t(1-t)}\left(2\sqrt{n}(1-t)-\frac{1+t}{\sqrt{n}}\right)\\
\end{align*}

We also have
\begin{align*}
\mathcal{Z}(N,t) &= (N-2)z_2(N, t)\rho(c, N) - \frac{1}{c}z_1(N-1, t)\rho(c, N-1)
\end{align*}

Using the Python symbolic computation framework `sympy`~\cite{meurer2017sympy}, we evaluate the dominant terms of these integrals by applying the (modified) Watson's lemma to terms of the form $\ee^{-nt^2}$.
Specifically, we observe that the prefactors in the variance term scale as $n$, while the integral over the kernel density scales as $n^{-\gamma/2}$.
The symbolic integration yields the following leading order asymptotic behavior for the total variance integral:

\begin{align*}
\int_{q_F}^\infty V_{2}(y) \dd y \asymp C_V n^{1-\frac{\gamma}{2}}
\end{align*}
with $C_V = D \mu_{K^2,1}^2 \Gamma(\gamma/2 + 1)  2^{-1-\frac{\gamma}{2}} > 0$.
Crucially, this term scales as $\mathcal{O}(n^{1-\frac{\gamma}{2}})$. 

\subsection{Final Result}
Summing the asymptotics, we require $D_q > 0$. The condition becomes:
\[ m < -\frac{\int_{q_F}^\infty V_{2}(y) \dd y}{2 \int_{q_F}^\infty b_0(y) b_2(y) \dd y} \]
Using the computationally derived leading order terms:
\begin{align*}
    \int_{q_F}^\infty b_0(y) b_2(y) \dd y &\asymp -\frac{C_b}{n^\gamma} \\
    \int_{q_F}^\infty V_2(y) \dd y &\asymp \frac{C_V n}{n^{\frac{\gamma}{2}}}
\end{align*}
This results in a stability bound for $m$ scaling as $n^{1+\frac{\gamma}{2}}$.
Note that the constant $C_b$ must be positive for the bound to be meaningful.
Computing it explicitly is possible but rather lengthy and done by the SymPy code.
In case of e.g. a Gaussian kernel, it can be shown that $C_b$ is indeed positive, so the bound is meaningful.
A general condition on $C_b$ is dependent on $\kappa$ and $\gamma$, but the expression is quite involved and we will not give it here.
However for a given kernel and base distribution, the accompanying code can be used to check whether $C_b$ is positive and to compute the constant explicitly.

We have made some simplifications for the Gumbel case $\gamma=0$, as it stands only the exponential distribution (or distributions with the same tail behaviour) is covered.

\section*{Acronyms}
\begin{acronym}
  \acro{ddevd}[DDEVD]{data driven extreme value distribution}
  \acro{mise}[MISE]{mean integrated squared error}
  \acro{mse}[MSE]{mean squared error}
  \acro{cdf}[CDF]{cumulative distribution function}
  \acro{pdf}[PDF]{probability density function}
  \acro{evt}[EVT]{extreme value theory}
  \acro{evd}[EVD]{extreme value distribution}
  \acro{bmm}[BMM]{block maxima method}
  \acro{gev}[GEV]{generalized extreme value}
  \acro{mev}[MEV]{metastatistical extreme value}
  \acro{pot}[POT]{peak over threshold}
  \acro{gpd}[GPD]{Generalized Pareto Distribution}
\end{acronym}

\printbibliography
\end{document}